\documentclass[nosumlimits,twoside]{amsart}
\usepackage{amsfonts, amsmath, amssymb}
\usepackage{graphicx}
\usepackage{float}
\usepackage{srcltx}
\usepackage[all]{xy}
\usepackage{version}
\usepackage[final]{showlabels} 
\usepackage[T1]{fontenc}
\usepackage{enumerate}
\usepackage[normalem]{ulem}
\usepackage{bm}
\usepackage{latexsym}
\usepackage[2emode]{psfrag}
\usepackage{yhmath}
\usepackage{array}
\usepackage{dsfont}
\usepackage{mathrsfs}
\usepackage{tikz-cd}
\usepackage{comment}
\usepackage{xcolor}
\usepackage{tabu}
\usepackage{makecell}

\newtheorem{theorem}{\sc Theorem}[section]
\newtheorem{proposition}[theorem]{\sc Proposition}
\newtheorem{lemma}[theorem]{\sc Lemma}
\newtheorem{corollary}[theorem]{\sc Corollary}
\theoremstyle{definition}
\newtheorem{definition}[theorem]{\sc Definition}
\newtheorem{example}[theorem]{\sc Example}
\newtheorem{remark}[theorem]{\sc Remark}

\allowdisplaybreaks
\excludeversion{invisible}

\newcommand{\chara}{\textrm{char }}

{
\left\{\begin{aligned}}
{\end{aligned}
\right.
}

\begin{document}


\title[$E(n)$-coactions on Semisimple Clifford Algebras]{$E(n)$-coactions on Semisimple Clifford Algebras}

\author{Fabio Renda}

\address{\parbox[b]{0.97\linewidth}{University of Ferrara, Department of Mathematics and Computer Science,\\ Via Machiavelli 30, 44121 Ferrara, Italy}}
\email{fabio.renda@unife.it}



\begin{abstract} In this article we prove that $E(n)$-coactions over a finite-dimensional algebra $A$ are classified by tuples $(\varphi, d_1, \ldots, d_n)$ consisting of an involution $\varphi$ and a family $(d_i)_{i=1,\ldots,n}$ of $\varphi$-derivations satisfying appropriate conditions. Tuples of maps can be replaced by tuples of suitable elements $(c, u_1, \ldots, u_n)$, whenever $A$ is a semisimple Clifford algebra.
\end{abstract}

\maketitle

\tableofcontents

\section{Introduction}
The study of actions and coactions of Hopf algebras on rings and associative algebras has been a matter of interest for mathematicians all around the world for at least four decades \cite{beco, CY, E, mas, ms}. 
Indeed, the problem of classifying coactions on an algebra is a natural generalization of the problem of classifying gradings on vector spaces, with many people in different areas of mathematics and physics sharing interest in this.

In our particular case the study of the coactions of the family of Hopf algebras $E(n)$ was initially motivated by the will to answer to a question suggested by C. Menini and B. Torrecillas and that originated from an example contained in \cite{mt}.
In the aforementioned paper the authors were concerned with the study of (h-)separable coalgebras in monoidal categories and gave non-trivial examples of (h-)separable coalgebras in the monoidal category $\mathcal{T}_{A \otimes H^{op}}^{\#}$ for $H=H_4$ the four-dimensional Sweedler's Hopf algebra and $A=Cl(\alpha,\beta,\gamma)$ a Clifford algebra. These particular coalgebras are called \emph{cowreath} and can be constructed in a natural way when we have a Hopf algebra $H$ and an $H$-comodule algebra $A$. 

The four-dimensional Clifford algebra $A=Cl(\alpha,\beta,\gamma)$ has a canonical $H_4$-comodule algebra structure that was described in \cite{PVO2}. The cowreath induced by this coaction is (h-)separable provided $\alpha, \beta$ and $\gamma$ satisfy the conditions described in \cite[Thm. 6.1]{mt}. Nevertheless, such conditions are forced by the particular form of the Casimir element used to realize (h-)separability and that was chosen in order to simplify calculations. 

At the time we did not know whether such a cowreath could be proved (h-)separable with no restriction whatsoever on the scalars $\alpha$, $\beta$ and $\gamma$ by just changing the $H$-comodule algebra structure of $A$ or the form of the Casimir element, so we decided to explore both possibilities. 
Menini and Torrecillas were able to prove in \cite[Thms. 1,2]{mt2} that if we keep on $A$ the canonical $H_4$-comodule algebra structure, the induced cowreath is always (h-)separable, though via a Casimir element of a much more general form. On the other hand, in order to understand how a different coaction would affect (h-)separability of the cowreath, we first needed to obtain a complete classification of $H_4$-coactions on $A=Cl(\alpha,\beta,\gamma)$. This turned out to be not an easy task, mostly because $A$ is not in general a semisimple algebra. 

Inspired by \cite{CY,ms} we understood that each $H_4$-coaction on a finite-dimensional algebra is completely determined by the choice of an involution $\varphi$ and a $\varphi$-derivation $d$ satisfying appropriate conditions, and that to obtain a full classification of $H_4$-coactions on a finite-dimensional algebra is equivalent to having a full perspective of its involutions and skew-derivations. Since any Hopf algebra $E(n)$ can be obtained from $H_4=E(1)$ by ``adjoining skew-primitive elements'' (see Definition~\ref{defEn}), this ultimately led us to the more general result herein contained (\autoref{maingen}), in which we classify all $E(n)$-coactions on a finite dimensional algebra. 

A complete classification of involutions and skew-derivations seems currently out of reach for a general Clifford algebra $A=Cl(\alpha,\beta_i,\gamma_i, \lambda_{ij})$, therefore we specialized our main result for the case when $A$ is semisimple (\autoref{oddcase}, \autoref{evencase}). In this instance a full classification of $E(n)$-coactions becomes roughly equivalent to the understanding of the structure of two particular subsets of $A$:
\[\mathcal{Z}^{\frac{1}{2}}(A)=\lbrace a \in A \ | \ a^2 \in \mathcal{Z}(A) \rbrace, \qquad \mathcal{Z}^{\sigma}(A)=\lbrace a \in A \ | \ a\sigma(a) \in \mathcal{Z}(A) \rbrace. \]
When $A=Cl(\alpha,\beta,\gamma)$ is a generalized quaternion algebra over $k$, $\mathcal{Z}^{\frac{1}{2}}(A)$ is known to be the union of the ground field $k$ and the $k$-space of so-called \emph{pure quaternions}. Nevertheless, when $A$ is not semisimple, coactions on $A$ can have a much more obscure presentation, due to the fact that its involutions and derivations are not necessarily all inner. This evidence prompted us to gather in this paper all the statements where the involved algebra $A$ has no fixed dimension, together with some interesting and useful results on Clifford algebras and their presentation by generators and relations. Further results related to the four-dimensional (non-semisimple) case and the answer to the question that started our research work are gathered in \cite{FR}.

\medskip

The paper is organized as follows. In Section \ref{preliminaries} we recall some basics on (co)-module algebras, involutions, derivations and Clifford algebras. In Section~\ref{Cliffordresults} we show that a Clifford algebra $A$ can be defined by generators and relations, by fixing a matrix $Q$ of scalars, and we show some useful relations occurring between these scalars and the properties of $A$. In particular we show that the semisimplicity of $A$ is equivalent to $\det Q \neq 0$ (Corollary~\ref{detQ}).  We also prove that the only Clifford algebras admitting a bialgebra structure are the algebras $E(n)$. Section~\ref{mainsec} contains all the steps to show that $E(n)$-coactions over a finite-dimensional algebra $A$ are classified by tuples $(\varphi, d_1, \ldots, d_n)$ consisting of an involution $\varphi$ and a family $(d_i)_{i=1,\ldots,n}$ of $\varphi$-derivations satisfying appropriate conditions (see Thm.~\ref{maingen}). In Section~\ref{gensemisimplecase} we determine an equivalent characterization for the case $A=Cl(\alpha,\beta_i,\gamma_i, \lambda_{ij})$ is a semisimple Clifford algebra. In the last section we include some examples in low dimension.
\medskip

\noindent\textit{Notations and conventions}. All vector spaces are understood to be over a fixed field $k$ and by linear maps we mean $k$-linear maps. $k^{\times}$ is used to indicate the multiplicative group of the field $k$. The unadorned tensor product $\otimes$ stands for $\otimes_k$. All linear maps whose domain is a tensor product will usually be defined on generators and understood to be extended by linearity. The direct product $R \times S$ of two rings $R$ and $S$ will always be equipped with componentwise operations. Algebras over $k$ will be associative and unital and coalgebras over $k$ will be coassociative and counital.
For an algebra, the multiplication and the unit are denoted by $m$ and $u$, respectively, while for a coalgebra, the comultiplication and the counit are denoted by $\Delta$ and $\varepsilon$, respectively. $S:H \rightarrow H$ will always denote the antipode map of $H$. We use the classical Sweedler’s notation 
and we shall write $\Delta(h)= h_{1}\otimes h_{2}$ for any $h\in H$, where we omit the summation symbol. $H^{cop}$ denotes the coopposite Hopf algebra, i.e. the Hopf algebra $H$ where the comultiplication is replaced by $\Delta^{cop}(h)=h_2 \otimes h_1$ for every $h \in H$. For a (right) $H$-comodule $M$ we write the coaction as $\rho: M\to M\otimes H$, $\rho(m)=m_0\otimes m_1$ for every $m \in M$. Given an Hopf algebra $H$ we will write ${_H}\mathrm{Vec}_k$ to denote the category of left $H$-modules and $\mathrm{Vec}_k^H$ to denote the category of right $H$-comodules. $H^*$ is the classical dual space of $H$. The center of an algebra $A$ is denoted by $\mathcal{Z}(A)$ and its group of units by $\textrm{U}(A)$, we write $J=Jac(A)$ to denote the Jacobson radical of $A$.

\section{Preliminaries}\label{preliminaries}

We start by recalling a couple of useful definitions.
\begin{definition}
Let $A$ be an algebra and let $H$ be a \emph{Hopf algebra}.
Given a $k$-linear map $\mu_A: H \otimes A \rightarrow A$, the pair $(A, \mu_A)$ is called a (left) \emph{$H$-module algebra} if the following conditions are satisfied.
\begin{eqnarray}
(A, \mu_A) &&  \textrm{is a (left)} \ H\textrm{-module.}\label{modalg0}\\
\mu(h \otimes ab)&=&\mu(h_1 \otimes a)\mu(h_2 \otimes b)\label{modalg1}\\
\mu(h \otimes 1_A)&=&\varepsilon(h)1_A \label{modalg2}
\end{eqnarray}
for all $h \in H $ and $a,b \in A$.
The map $\mu_A$ will also be called a (left) $H$-action on $A$.
\end{definition}
\begin{definition}
Let $A$ be an algebra and let $H$ be a Hopf algebra. Given a $k$-linear map $\rho_A: A \rightarrow A \otimes H$, the pair $(A, \rho_A)$ is called a (right) \emph{$H$-comodule algebra} if the following conditions are satisfied.
\begin{eqnarray}
(A, \rho_A) &&  \textrm{is a (right)} \ H\textrm{-comodule.}\label{comodalg1}\\
\rho_A (ab) &=& a_0b_0 \otimes a_1b_1.\label{comodalg2}\\
\rho_A (1_A)&=&1_A \otimes 1_H. \label{comodalg3}
\end{eqnarray}
for all $a,b \in A$.
The map $\rho_A$ will also be called a (right) $H$-coaction on $A$.
\end{definition}

\begin{remark} An $H$-module algebra is an algebra in the category ${_H}\mathrm{Vec}_k$ of left $H$-modules. An $H$-comodule algebra is an algebra in the category $\mathrm{Vec}_k^H$ of left $H$-comodules.
\end{remark}

\begin{definition}
Let $A$ be an algebra and $\varphi:A \rightarrow A$ an algebra endomorphism. We will call $\varphi$ an \emph{algebra involution} (or simply an \emph{involution}) if $\varphi^2=\textrm{Id}_A$. 

\medskip

Fix an algebra map $\varphi:A \rightarrow A$. A $k$-linear map $d:A \rightarrow A$ such that
\[d(ab)=d(a)b+\varphi(a)d(b)\]
for every $a, b \in A$ is called a $\varphi$-\emph{derivation} (or a \emph{skew-derivation}).
\end{definition}

\subsection{Clifford algebras}
Some of the main results contained in this paper concern Clifford algebras. We recall their definition and some of their important properties.

Consider a quadratic form $q: V \rightarrow k$ on a vector space $V$ over a field $k$.
This means that $q$ satisfies
\begin{eqnarray*}
&q(\lambda v)=\lambda^2 q(v) \quad \textrm{for any } \lambda \in k \textrm{ and any } v \in V,\\
&\textrm{The mapping }\beta_q(u,v):=q(u+v)-q(u)-q(v) \textrm{ is bilinear}.
\end{eqnarray*}
We can build the tensor algebra $T(V)$ over $V$ and consider the ideal 
$I_q$ generated by elements
\[v \otimes v - q(v)1\]
for $v \in V$.
The Clifford algebra $Cl(V,q)$ associated to $V$ and $q$ is the quotient algebra $\frac{T(V)}{I_q}$.

Let $\iota : V \rightarrow Cl(V,q)$ be the map defined by the composition of the inclusion of $V$ into $T(V)$ with the canonical projection $\pi :T(V) \rightarrow Cl(V,q)$. Clifford algebras satisfy the following universal property.
\begin{theorem}\cite[Theorem~3.1]{C}\label{Unicliff}
Let $A$ be a $k$-algebra and $f:V \rightarrow A$ be a $k$-linear map such that $(f(v))^2=q(v)1_A$ for all $v \in V$. Then there exists an algebra morphism $\varphi:Cl(V,q) \rightarrow A$ such that
\[\varphi(\iota(v))=f(v) \qquad \textrm{for all }v \in V. \]
\end{theorem}
If $V$ has finite dimension $n$ and $k$ has characteristic $\chara  k \neq 2$, it is known that $V$ admits an orthogonal basis with respect to $\beta_q( \cdot, \cdot)$, i.e. a $n$-uple of vectors $(e_1, e_2, \ldots, e_n)$ such that 
\[\beta_q(e_i, e_j)=0, \ j \neq i, \qquad \beta_q(e_i, e_i)=q(e_i), \qquad \textrm{ for every } i =1, \ldots, n.\]
It is also known that in this case a basis for $Cl(V,q)$ is given by the linearly independent elements $e_{i_1}\cdots e_{i_h}$ with $i_1< \ldots < i_h$ and therefore that $\dim_k Cl(V,q)=2^n$. The distinguished element $z=e_1e_2 \cdots e_n$ is called the \emph{pseudoscalar} of $A$. Clifford algebras are ($\mathbb{Z}_2$-)graded algebras, i.e. they admit a (unique) decomposition $A=A_0 \oplus A_1$ into an even part $A_0$ and an odd part $A_1$. An element $e_{i_1}\cdots e_{i_h}$ of the basis is in $A_m$ if, and only if $h \equiv m \ (\textrm{mod } 2)$.  For a further insight into Clifford algebras we refer to \cite{C,La,Lo}.

\section{First results on Clifford algebras}\label{Cliffordresults}

When specialized to the case $k=\mathbb{R}, \mathbb{C}$, some of the results contained in this section might be considered well-known facts by Clifford algebra experts. Nonetheless, we decided to include proofs for every statement, as we could not find a detailed reference for these exact phrasings (expecially for the ones listed in the first subsection). We would like to outline the fact that in all of them no assumption is made on the ground field $k$, except from $\chara k \neq 2$.

\subsection{Clifford-type algebras and (classical) Clifford algebras}\label{ortcliff}

In \cite{PVO2} a family of algebras called ``of Clifford-type'' were introduced, in order to describe $H_4$-cleft extensions of the ground field $k$, where $H_4$ denotes Sweedler's Hopf algebra. 
\begin{definition}\cite[Def.~$1$]{PVO2}\label{clt}
Let $\alpha, \beta_i,\gamma_i \in k$ for $i=1, \ldots, n$ and $\lambda_{ij} \in k$ for $i, j \in \lbrace 1, \ldots, n \rbrace$, $i<j$.
The \emph{Clifford-type algebra} $Cl(\alpha,\beta_i,\gamma_i, \lambda_{ij})$ is the unital associative algebra generated by elements $G, X_1, \ldots, X_n$ such that $G^2=\alpha$, $X_i^2=\beta_i$, $GX_i+X_iG=\gamma_i$ for all $i=1, \ldots, n$ and $X_iX_j+X_jX_i=\lambda_{ij}$ for all $i,j \in \lbrace 1, \ldots, n \rbrace$ with $i<j$. A $k$-basis for this algebra is given by $\lbrace G^j X_{P} \rbrace$, where $j=0,1$ and $X_P=X_{i_1} \cdots X_{i_s}$ with $P=\lbrace i_1 < i_2 < \ldots < i_s \rbrace \subseteq \lbrace 1, \ldots, n \rbrace$. 
\end{definition}
These Clifford-type algebras are (isomorphic to) classical Clifford algebras whose presentation is given by generators and relations.
Let us consider the Clifford-type algebra $Cl(\alpha, \beta_i, \gamma_i, \lambda_{ij})$ and the free vector space $V$ with basis $(e_0, e_1, \dots, e_n)$. We define
\begin{equation}\label{Qmatrix}
Q:=\begin{pmatrix}
\alpha & \frac{\gamma_1}{2} & \cdots & & \cdots & \frac{\gamma_n}{2}\\
\frac{\gamma_1}{2} & \beta_1 & \cdots & & \cdots & \frac{\lambda_{2n}}{2}\\
\vdots  & \vdots & \ddots & & & \\
& & & & & \vdots\\
\vdots & \vdots & & & \beta_{n-1} & \frac{\lambda_{{n-1}n}}{2}\\
\frac{\gamma_n}{2} & \frac{\lambda_{2n}}{2} & &  \cdots & \frac{\lambda_{{n-1}n}}{2} & \beta_n
\end{pmatrix}
\end{equation}
the symmetric matrix associated to the quadratic form $q: V \rightarrow k$ such that $q(v)=v^tQv$ for every $v \in V$.
We have $q(e_0)=\alpha$ and $q(e_i)=\beta_i$ for all $i=1 \ldots n$, while $\beta_q(e_0, e_i)=\gamma_i$ for all $i=1, \ldots n$ and $ \beta_{q}(e_i, e_j)=\lambda_{ij}$ for all $i,j \in \lbrace 1, \ldots, n \rbrace$.
Whenever $\gamma_i=0=\lambda_{ij}$ for all $i, j \in \lbrace 1, \ldots, n \rbrace$ we find the classical presentation of a Clifford algebra with orthogonal generators:
\[Cl(\alpha, \beta_i, 0, 0)\cong Cl(V,q).\]
The explicit isomorphism can be presented in the following way. Define a $k$-linear map $\Phi: V \rightarrow Cl(\alpha, \beta_i, 0, 0)$ by setting
\[\Phi(e_0)=G, \ \Phi(e_i)=X_i \quad \textrm{for every } i=1, \ldots, n.\]
This map can be extended in a unique way to an algebra map whose domain is $T(V)$, by the fundamental property of the tensor algebra. Furthermore
\begin{eqnarray*}
\Phi(v \otimes v -q(v)1)&=&\Phi(v \otimes v)-q(v)\\
&=&(\Phi(v))^2-q(v)\\
&=&\left(\sum_{i=0}^n\lambda_i \Phi(e_i)\right)^2-\sum_{i=0}^n\lambda_i^2 q(e_i)\\
&=&\left(\lambda_0G+\sum_{i=1}^n\lambda_i X_i \right)^2-\sum_{i=0}^n\lambda_i^2 q(e_i)\\
&\overset{\gamma_i=\lambda_{ij}=0}{=}&\lambda_0^2 \alpha +\sum_{i=1}^n\lambda_i^2 \beta_i-\sum_{i=0}^n\lambda_i^2 q(e_i)\\
&=&0,
\end{eqnarray*}
for every $v=\sum_{i=0}^n\lambda_i e_i$. This means $\Phi(I_q)=0$, i.e. $\Phi$ induces an algebra map $\overline{\Phi}$ between $Cl(V,q)$ and $Cl(\alpha, \beta_i, 0, 0)$. Since $\overline{\Phi}(e_0^{i_0}e_1^{i_1} \cdots e_n^{i_n})=G^{i_0}X_1^{i_1} \cdots X_n^{i_n}$ for $i_j \in \lbrace 0,1 \rbrace$, we see that $\overline{\Phi}$ sends a basis into a basis and thus that it is invertible. 

On the other hand, if at least one of the $\gamma_i$'s or $\lambda_{ij}$'s is not zero, then we need the matrix $Q$ to be orthogonally diagonalizable. It is actually a known fact that a symmeytric matrix with entries \emph{in a field with characteristic different from $2$} can be diagonalized or equivalently, that we can always find a basis of $V$ which is orthogonal with respect to the bilinear form $\beta_q$. Proofs of this result are usually carried out by induction and are not constructive (see e.g. \cite[pp. $114$-$115$]{St}). Once we have obtained an orthogonal basis $(v_0, v_1, \ldots, v_n)$, we have determined a new presentation of our Clifford type-algebra
\[Cl(\alpha', \beta'_i, 0, 0)\ \cong Cl(\alpha, \beta_i, \gamma_i, \lambda_{ij}).\]
Then we can proceed as before and determine the explicit isomorphism $\overline{\Phi}: Cl(\alpha', \beta'_i, 0, 0) \rightarrow  Cl(V,q)$.
We can conclude that the following statement holds.

\begin{theorem}\label{ortbas} Let $A=Cl(\alpha, \beta_i, \gamma_i, \lambda_{ij})$ be a Clifford-type algebra on a field $k$ with $\chara k \neq 2$. Then there exists a (classical) Clifford algebra $Cl(V,q)$ which is isomorphic to $A$. To find a presentation that makes this clear, one must find an orthogonal basis for the quadratic form $q: V \rightarrow k$ defined via the matrix $Q$ in \eqref{Qmatrix}.
\end{theorem}

Theorem~\ref{ortbas} guarantees that any Clifford-type algebra can be regarded as a (classical) Clifford algebra where the generators have not been (necessarily)  chosen to be orthogonal. This enable us to use both the presentation of $A$ by generators and relations (i.e. Definition~\ref{clt}) and the vast literature on Clifford algebras without making a distinction depending on the values of the $\gamma_i$'s and the $\lambda_{ij}$'s.
In the next subsection we will explain how the semisimplicity of a Clifford algebra $A=Cl(\alpha, \beta_i, \gamma_i, \lambda_{ij})$ is related to its associated quadratic form $q$, or, equivalently, to the values of its defining scalars.

\subsection{Semisimplicity}

The following theorem gives a complete classification of semisimple Clifford algebras and also presents information on the structure of the Jacobson radical in the non-semisimple case. 

\begin{theorem}\cite{Sh}\label{semisimple}
Consider a Clifford algebra $A=Cl(\alpha, \beta_i,0,0)$ on a field $k$ with $\chara k \neq 2$ and let $J=Jac(A)$ be its Jacobson radical.
\begin{itemize}
\item If $\alpha =0$, then $G \in J$ and similarly if $\beta_i=0$, then $X_i \in J$.
\item If $\alpha \neq 0$ and $\beta_i \neq 0$ for all $i=1, \ldots n$, then $J=0$.
\item $A$ is semisimple if and only if $\alpha \neq 0$ and $\beta_i \neq 0$ for all $i=1, \ldots, n$.
\end{itemize}
\end{theorem}

\begin{proof}
First of all, observe that $\alpha=0$ means that $G^2=0$ and this is equivalent to the fact that for any $a \in A$ the element $1-aG$ is invertible and its inverse is $1+aG$. This is equivalent to $G \in J$. The proof for the statement with $X_i$ is identical. 

Now let $\mathcal{B}=\lbrace G^j X_{P} \rbrace$ be the $k$-basis of $A$ fixed in Definition~\ref{clt} and suppose $\alpha \neq 0$ and $\beta_i \neq 0$ for all $i=1, \ldots n$. Note that in this way $G$, each $X_i$, and hence each element of $\mathcal{B}$, are invertible elements of $A$. Consider the $k$-algebra homomorphism $\rho : A \rightarrow \textrm{End}_k(A)$ defined by $\rho(a)(z)=az $ for all $a,z \in A$. Define the map $\Phi : A \rightarrow k$ by $\Phi(a) = \textrm{tr}(\rho(a)), \ a \in A$, where $\textrm{tr}(\rho(a))$ is the trace of the matrix corresponding to the linear map $\rho(a)$ with respect to the ordered basis $\mathcal{B}$. We now make three simple observations.

\begin{enumerate}
\item $\Phi(1)=|\mathcal{B}|=2^{n+1}$ because $\rho(1)$ is the identity map of $A$.
\item If $1 \neq a \in \mathcal{B}$, then $\Phi(a)=0$. That’s because $\rho(a)(z)=az \neq z$ for all $z \in \mathcal{B} \setminus \lbrace a \rbrace$ and so the diagonal entries of the matrix of $\rho(a)$ are all zero hence $\Phi(a)=\textrm{tr}(\rho(a))=0$.
\item If $a \in A$ is nilpotent, then $\Phi(a)=0$. That’s because $a^m=0$ for some $m$ and so $(\rho(a))^m = \rho(a^m)=0$. Thus $\rho(a)$ is nilpotent and we know that the trace of a nilpotent matrix is zero. 
\end{enumerate}

Let $a \in J$. Then $a \neq 1$ and, since $A$ is Artinian, $a$ is nilpotent, hence $\Phi(a)=0$, by $3)$. Let $a = \sum_{i=1}^{2^{n+1}} c_i z_i$, where $c_i \in k$, $z_i \in \mathcal{B}$, $z_1=1$. So, by $1)$, $ 2)$,
\[0=\Phi(a)=\sum_{i=1}^{2^{n+1}} c_i \Phi(z_i)=c_1\Phi(z_1)=c_1\Phi(1)=2^{n+1}c_1\]
and hence $c_1=0$ because $\chara k \neq 2$. So the coefficient of $z_1$ of every element in $J$ is zero. But for every $i$, the coefficient of $z_1=1$ of the element $z_i^{-1}a \in J$ is $c_i$ and so $c_i=0$ for all $i$ (recall that the $z_i$'s are invertibles). Hence $a = 0$ and so $J=0$. 

Finally, recall that a ring is semisimple if and only if it is Artinian and its Jacobson radical is zero.
\end{proof}

\begin{corollary}\label{detQ}
Consider a Clifford algebra $A=Cl(\alpha, \beta_i,\gamma_i,\lambda_{ij})$ on a field $k$ with $\chara k \neq 2$ and let $Q$ be the matrix defined in \eqref{Qmatrix}.
Then $A$ is semisimple if, and only if, $\det Q \neq 0$.
\end{corollary}
\begin{proof}
Thanks to Theorem~\ref{ortbas} we know that, given a Clifford algebra $A=Cl(\alpha, \beta_i, \gamma_i, \lambda_{ij})$ whose associated quadratic form $q$ has matrix $Q$ with respect to the standard basis, we can always find a basis $\mathcal{B}=(v_0,v_1, \ldots, v_n)$ so that $q$ can be represented by a diagonal matrix 
\begin{equation*}
Q'=\begin{pmatrix}
\alpha' & 0 & \cdots & & \cdots &0\\
0 & \beta'_1 & \cdots & & \cdots & 0\\
\vdots  & \vdots & \ddots & & & \\
& & & & & \vdots\\
\vdots & \vdots & & & \beta'_{n-1} & 0\\
0 & 0 & &  \cdots & 0 & \beta'_n
\end{pmatrix}
\end{equation*}
with respect to $\mathcal{B}$.
It is also well-known that $\det Q$ and $\det Q'$ differ by a non-zero square scalar ($Q$ and $Q'$ are congruent) and therefore that $\det Q \neq 0$ if, and only if, $\det Q' \neq 0$.
\end{proof}
\begin{example}\label{4dim}
If we take a four-dimensional Clifford algebra $A=Cl(\alpha, \beta, \gamma)$, then $A$ is semisimple if, and only if, $\det Q =\alpha \beta - \frac{\gamma^2}{4} \neq 0$.
\end{example}
Corollary~\ref{detQ} agrees with the detailed classification of semisimple Clifford algebras reported in \cite{La}.
\begin{theorem}\cite[Thms. V.2.4 - V.2.5]{La}\label{class1}
Let $A$ be a $2^{n+1}$-dimensional Clifford algebra on a field $k$ with $\chara k \neq 2$ and associated quadratic form $Q$. Assume $\delta:=(-1)^{\frac{n(n+1)}{2}}\det Q \neq 0$.
\begin{itemize}
\item If $n$ is odd, then $\mathcal{Z}(A)=k$ and $A$ is a central simple algebra over $k$.
\item If $n$ is even and $\delta \notin k^2$, then $\mathcal{Z}(A)=k(\sqrt{\delta} )$ and $A$ is a central simple algebra over $k(\sqrt{\delta})$.
\item If $n$ is even and $\delta \in k^2$, then $\mathcal{Z}(A)=k \times k$ and $A \cong A_0 \times A_0$, where $A_0$ is the even part of $A$. Since $A_0$ is a Clifford algebra and $\dim_k A_0=2^n$, $A_0$ is central simple over $k$ and thus $A$ is semisimple over $k$.
\end{itemize}
\end{theorem}

Theorem~\ref{semisimple} can be generalized to the case of a Clifford algebra $A=Cl(\alpha, \beta_i, \gamma_i, \lambda_{ij})$, where the generators of $A$ are not necessarily pairwise orthogonal.

\begin{theorem}\label{jacobson}
Let $A=Cl(\alpha, \beta_i, \gamma_i, \lambda_{ij})$ be a Clifford algebra on a field $k$ with $\chara k \neq 2$. Let $Q$ be the matrix defined in \eqref{Qmatrix} and $q:V \rightarrow k$ the associated quadratic form on the generating space $V$, so that $A\cong Cl(V,q)$.
Then there is an algebra isomorphism
\[\frac{Cl(V,q)}{(\ker Q)} \cong Cl(\overline{V}, \overline{q}),\]
where $\overline{V}=\frac{V}{\ker Q}$ and $\overline{q}$ is the (non-degenerate) quadratic form induced by $q$ on $\overline{V}$.
Moreover
\[J(A)=(\ker Q)\]
and the Clifford algebra $Cl(\overline{V},\overline{q})$ is semisimple.
\end{theorem}

\begin{proof}
Let us fix the vector space $V=\lbrace G, X_1, \ldots, X_n \rbrace$ and the usual matrix $Q$ associated to the quadratic form $q:V \rightarrow k$.  We write $q(v)=v^tQv$ for all $v \in V$, so that $A=Cl(\alpha, \beta_i, \gamma_i, \lambda_{ij}) \cong Cl(V,q)$.
Remember that there is a symmetric bilinear form associated to $q$ defined as $\beta_q(v,w)=q(v+w)-q(v)-q(w)$ for all $v,w \in V$. The kernel of $\beta_q$ (also called the radical of $\beta_q$) is just the kernel of the matrix $Q$.
The bilinear form $\beta_q$ induces a non-degenerate bilinear form $\overline{\beta_q}:\overline{V} \times \overline{V} \rightarrow k$ on $\overline{V}:=\frac{V}{\ker Q}$, defined by
\[\overline{\beta_q}(\overline{v}, \overline{w}):=\beta(v,w), \quad \textrm{ for every } \overline{v}=v+\ker Q, \ \overline{w}=w+\ker Q \in \overline{V}.\]
Then clearly we can consider the Clifford algebra $Cl(\overline{V},\overline{q})$, where $\overline{q}:\overline{V} \rightarrow k$ is simply the quadratic form associated to $\overline{\beta_q}$, defined by $\overline{q}(\overline{v})=\frac{1}{2}\overline{\beta_q}(\overline{v},\overline{v})$ for every $\overline{v} \in \overline{V}$.

Now consider the diagram
\begin{equation*}
 \begin{tikzcd}
    V \arrow{r}{\iota} \arrow[swap]{d}{p_{\ker Q}} & Cl(V,q) \arrow[dashed]{d}{\exists!   \ \overline{p}}\\
     \frac{V}{\ker Q} \arrow{r}{\overline{\iota}} & Cl(\overline{V},\overline{q})
  \end{tikzcd}.
\end{equation*}
where $\iota$ and $\overline{\iota}$ are the canonical inclusions of each generating vector space into the respective Clifford algebra.
Since we have
\[(\overline{\iota} \circ p_{\ker Q}(v))^2=(\overline{v})^2=\overline{q}(\overline{v})=\frac{1}{2}\overline{\beta_q}(\overline{v},\overline{v})=\frac{1}{2} \beta_q(v,v)=q(v),\]
by the universal property of Clifford algebras (Theorem~\ref{Unicliff}), there exists a unique algebra morphism $\overline{p}:Cl(V,q) \rightarrow Cl(\overline{V}, \overline{q})$, such that $\overline{\iota} \circ p_{\ker Q}=\overline{p} \circ \iota$.
From this, with abuse of notation, we deduce that $\overline{p}(v)=\overline{v}$ for every $v \in V$. 

Now we can consider an element of the basis of $Cl(\overline{V},\overline{q})$, say $\overline{z}=\overline{v_{i_1}}\overline{v_{i_2}}\cdots\overline{v_{i_r}}$, for some generators $\overline{v_{i_1}},\overline{v_{i_2}}, \ldots,\overline{v_{i_r}} \in \overline{V}$, and show that
\[z=\overline{v_{i_1}}\overline{v_{i_2}}\cdots\overline{v_{i_r}}=\overline{p}(v_{i_1})\overline{p}(v_{i_2})\cdots\overline{p}(v_{i_r})=\overline{p}(v_{i_1}v_{i_2}\cdots v_{i_r}).\]
Since this holds for any element $z$ of the basis of $Cl(\overline{V},\overline{q})$, this means that $\overline{p}$ is a surjective morphism of algebras.

Next, consider the ideal $(\ker Q)$ generated in $Cl(V,q)$. If $a \in (\ker Q)$, then $a=\sum_i s_iq_i$ with $s_i \in Cl(V,q)$ and $q_i \in \ker Q$.
Then we have
\[\overline{p}(a)=\sum_i \overline{p}(s_i)\overline{p}(q_i)=\sum_i \overline{p}(s_i)(\overline{\iota}p_{\ker Q})(q_i)=\sum_i \overline{p}(s_i) \cdot 0=0,\]
which means that $(\ker Q) \subseteq \ker \overline{p}$.
Let us indicate once again with $\mathcal{B}=\lbrace G^jX_P \rbrace$ the usual basis of $Cl(V,q)$. Then, if we pick an element $a$ in $\ker \overline{p}$, we can write it $a=\sum_{i=1}^{2^{n+1}}c_iz_i$, with $c_i \in k$, $z_i \in \mathcal{B}$, $z_1=1$. Clearly we have
\[0=\overline{p}(a)=\sum_{i=1}^{2^{n+1}}c_i\overline{p}(z_i)=\sum_{i=1}^{2^{n+1}}c_i\overline{p}(v_{i_1})\overline{p}(v_{i_2})\cdots \overline{p}(v_{i_r}),\]
where we have just split every element $z_i$ of the basis into a product of generators contained in $V$.
We can use the fact that  $\overline{\iota} \circ p_{\ker Q}=\overline{p} \circ \iota$ to write
\[0=\sum_{i=1}^{2^{n+1}}c_i\overline{p}(v_{i_1})\overline{p}(v_{i_2})\cdots \overline{p}(v_{i_r})=\sum_{i=1}^{2^{n+1}}c_i\overline{v_{i_1}}\overline{v_{i_2}}\cdots \overline{v_{i_r}}.\]
Each product $\overline{v_{i_1}}\overline{v_{i_2}}\cdots \overline{v_{i_r}}$ is either $0$ (when at least one of the $v_{i_j}$'s is in $\ker Q$) or an element of the basis of $Cl(\overline{V},\overline{q})$, therefore we find that $c_i=0$ if $z_i \notin (\ker Q)$, so that $a \in (\ker Q)$. From this we can conclude that
\[\frac{Cl(V,q)}{(\ker Q)} \cong Cl(\overline{V}, \overline{q}).\]

Moreover, if we denote $A:=Cl(V,q)$ and $\overline{A}:=Cl(\overline{V},\overline{q})$, then, thanks to the surjectivity of $\overline{p}:A \rightarrow \overline{A}$, we have
\[\overline{p}(J(A)) \subseteq J(\overline{A}).\]
But $\overline{A}$ is a semisimple Clifford algebra, because $\overline{q}$ is non-degenerate (see Corollary~\ref{detQ}), thus $J(\overline{A})=0$ and this immediately implies $J(A) \subseteq \ker \overline{p}=(\ker Q)$. Finally, consider $a \in \ker Q$. For any $v \in V$ we have
\[av+va=a^tQv+v^tQa=0 \ \textrm{and } a^2=0,\]
where we have identified $a$ and $v$ with their coordinate vectors on the usual basis in $A$, with abuse of notation. If we consider $z \in \mathcal{B}=\lbrace G^jX_P \rbrace$, such that $z=v_{i_1}v_{i_2} \cdot v_{i_r}$, then it is clear that
\[az=av_{i_1}v_{i_2} \cdot v_{i_r}=-v_{i_1}av_{i_2} \cdot v_{i_r}=v_{i_1}v_{i_2}av_{i_3}\cdots v_{i_r}=(-1)^rza,\]
i.e. $az+(-1)^{r+1}za=0$ for any $z \in \mathcal{B}$.
Then, for every element $A \ni b=\sum_{i=1}^{2^{n+1}}\alpha_{j,P}G^jX_P$, we have that 
\begin{eqnarray*}
(1-ba)(1+ba)&=&1-ba+ba+baba\\
&=&1+ba\left( \sum_{i=1}^{2^{n+1}}\alpha_{j,P}G^jX_P\right)a\\
&=&1+ba^2\left( \sum_{i=1}^{2^{n+1}}(-1)^{j+|P|}\alpha_{j,P}G^jX_P\right)\\
&\overset{a^2=0}{=}&1,
\end{eqnarray*}
i.e. $1-ba$ is invertible for any $b \in A$.
This is equivalent to $a \in J(A)$, and therefore we can conclude that
\[J(A)=(\ker Q).\]
\end{proof}
%
%
%

\subsection{Hopf algebra structures on Clifford algebras}

In general, Clifford algebras are not Hopf algebras. The only members of this family who admit an Hopf algebra structure are those isomorphic (as associative algebras) to $E(n):=Cl(1,0,0,0)$. This family of Hopf algebras generalizes Sweedler's Hopf algebra $H_4=E(1)$ and were introduced in \cite[p. $755$]{BDG} and studied in \cite{CC, CD, PVO, PVO2}.

\begin{definition}{\cite[p.$18$]{CD}}\label{defEn}
We denote by $E(n)$ the $2^{n+1}$-dimensional Hopf algebra over a field $k$ of characteristic $\chara k \neq 2$ generated by elements $g$ and $x_i$, for $i=1, \ldots, n$, such that $g^2=1$, $x_i^2=0$  and $gx_i=-x_ig$ for any $i=1, \ldots, n$ and $x_ix_j=-x_jx_i$ for $i,j=1, \ldots, n$, $i <j$.
\end{definition}
The canonical Hopf algebra structure of $E(n)$ is given by
\[\Delta(g)=g \otimes g, \quad \Delta(x_i)=x_i \otimes g +1 \otimes x_i, \ i=1, \ldots, n\]
\[\varepsilon(g)=1, \quad \varepsilon(x_i)=0, \ i=1, \ldots, n\] 
\[S(g)=g^{-1}=g, \quad S(x_i)=gx_i, \ i=1, \ldots, n.\]
Let us prove that these are essentially the only Clifford algebras with a Hopf structure.
\begin{theorem}\label{hopfcliff}
Let $A=Cl(\alpha, \beta_i, \gamma_i, \lambda_{ij})$ be a $2^{n+1}$-dimensional Clifford algebra on a field $k$ with $\chara k \neq 2$. Then $A$ is a bialgebra if, and only if, it is isomorphic to the algebra $E(n)=Cl(1, 0,0,0)$ (not necessarily as bialgebras).
\end{theorem}
\begin{proof}
Consider $A=Cl(\alpha, \beta_i, \gamma_i, \lambda_{ij})$ a Clifford algebra on a field $k$ with $\chara k \neq 2$. As stated in Theorem~\ref{ortbas}, we can always assume that $\gamma_i=0$ and $\lambda_{ij}=0$ for all $i,j \in \lbrace 1, \ldots, n \rbrace$, $i<j$. Suppose we define a coalgebra structure $(A,\Delta, \varepsilon)$ on $A$, such that $A$ becomes a bialgebra. In this case $\varepsilon$ is an algebra map and we have
\[0=\gamma_i=\varepsilon(\gamma_i)=\varepsilon(GX_i+X_iG)=\varepsilon(GX_i)+\varepsilon(X_iG)=2 \varepsilon(G)\varepsilon(X_i)\]
and
\[0=\lambda_{ij}=\varepsilon(\lambda_{ij})=\varepsilon(X_iX_j+X_jX_i)=\varepsilon(X_iX_j)+\varepsilon(X_jX_i)=2 \varepsilon(X_i)\varepsilon(X_j)\]
for every $i,j \in \lbrace 1, \ldots, n \rbrace$, $i<j$.

Now suppose there exists an $r \in \lbrace 1, \ldots, n \rbrace$ such that $\varepsilon(X_r) \neq 0$. Then $2\varepsilon(G)\varepsilon(X_r)=0$ and $2 \varepsilon(X_r)\varepsilon(X_j)=0$ for every $j = 1, \ldots, n $, $j \neq r$ imply that $\varepsilon(G)=0$ and $\varepsilon(X_j)=0$ for every $j = 1, \ldots, n$, $j \neq r$. By swapping the roles of $G$ and $X_r$ we can suppose that $\varepsilon(G) \neq 0$ and $\varepsilon(X_i)=0$ for every $i=1, \ldots, n$. Hence we obtain that $\alpha=\varepsilon(\alpha)=\varepsilon(G^2)=\varepsilon(G)^2 \neq 0$ and $\beta_i=\varepsilon(\beta_i)=\varepsilon(X_i^2)=\varepsilon(X_i)^2=0$ for every $i=1, \ldots, n$. In this case we see that $A$ can be identified with $Cl(\alpha, 0,0,0)$ with $\alpha \neq 0$.

On the other hand, if $\varepsilon(G)=\varepsilon(X_i) = 0$ for all $i=1, \ldots, n$, then we can call $X_{n+1}:=G$ and we get that
\[\alpha=\varepsilon(\alpha)=\varepsilon(X_{n+1}^2)=\varepsilon(X_{n+1})^2=0\]
and
\[\beta_i=\varepsilon(\beta_i)=\varepsilon(X_i^2)=\varepsilon(X_i)^2=0 \quad  \textrm{ for every } i=1, \ldots, n,\]
i.e. $A=Cl(0,0,0,0)$ is the $2^{n+1}$-dimensional Exterior algebra. We are going to prove that the latter case does not admit a bialgebra structure.
Let us write 
\[\Delta(X_1)=\sum_{P,Q \subseteq \lbrace 1, \ldots n+1 \rbrace}c_{P,Q} X_P \otimes X_Q,\]
with $c_{P,Q} \in k$ and $X_R=X_{i_1} \cdots X_{i_{|R|}}$ with $R=\lbrace i_1 < i_2 < \ldots < i_{|R|}\rbrace \subseteq \lbrace 1, \ldots, n+1 \rbrace$. Then the counit axioms give
\[X_1=(A \otimes \varepsilon)\Delta(X_1)=\sum_{P,Q \subseteq \lbrace 1, \ldots n+1 \rbrace}c_{P,Q} X_P \varepsilon(X_Q)=\sum_{P \subseteq \lbrace 1, \ldots n+1 \rbrace}c_{P,\emptyset} X_P\]
and
\[X_1=(\varepsilon \otimes A)\Delta(X_1)=\sum_{P,Q \subseteq \lbrace 1, \ldots n+1 \rbrace}c_{P,Q}\varepsilon(X_P) X_Q=\sum_{Q \subseteq \lbrace 1, \ldots n+1 \rbrace}c_{\emptyset, Q} X_Q.\]
These force $c_{\emptyset, \emptyset}=0$, $c_{\lbrace 1 \rbrace, \emptyset}=1$, $c_{\emptyset, \lbrace 1 \rbrace}=1$ and $c_{P,\emptyset}=0$, $c_{\emptyset, Q}=0$ for every $\emptyset \neq P,Q \neq \lbrace 1 \rbrace$, i.e.
\[\Delta(X_1)= 1 \otimes X_1+X_1 \otimes 1+\sum_{\emptyset \neq P,Q \subseteq \lbrace 1, \ldots n+1 \rbrace}c_{P,Q} X_P \otimes X_Q.\]
If we square this equality we find
\begin{eqnarray*}
\Delta(X_1)^2&=&2X_1 \otimes X_1+\sum_{\emptyset \neq P,P',Q,Q' \subseteq \lbrace 1, \ldots n+1 \rbrace}c_{P,Q}c_{P',Q'} X_PX_{P'} \otimes X_QX_{Q'}+\\
&+&\sum_{\emptyset \neq P,Q \subseteq \lbrace 1, \ldots, n+1 \rbrace} c_{P,Q} [X_P \otimes (X_1X_Q+X_QX_1)+(X_1X_P+X_PX_1) \otimes X_Q].
\end{eqnarray*}
Notice that $2X_1 \otimes X_1$ and any other term appearing in the RHS are linearly independent, since the former has no tensorand that is contained in $(\ker \varepsilon)^2$. On the other hand $\Delta(X_1)^2=\Delta(X_1^2)=\Delta(\beta_1)=\Delta(0)=0$, which is a contradiction. This proves that the exterior algebra $A=Cl(0,0,0,0)$ does not admit any bialgebra structure.

Therefore any Clifford algebra $A$ admitting a bialgebra structure must be isomorphic to a Clifford algebra of the form $Cl(\alpha, 0,0,0)$ with $\alpha \neq 0$. Furthermore, if $k$ does not contain a square root of $\alpha$ we cannot define a bialgebra structure on $A$, because $\varepsilon(G)$ cannot be defined (remember that $\varepsilon(G)^2=\alpha$). Conversely, if $\sqrt{\alpha} \in k$, then we can substitute the generator $G$ with $\frac{G}{\sqrt{\alpha}}$ and we get that $A$ is isomorphic to the algebra $Cl(1,0,0,0)=E(n)$.
\end{proof}

\section{\texorpdfstring{$E(n)$}{E(n)}-coactions on finite-dimensional algebras}\label{mainsec}

In this section we explain in detail the steps followed to classify all $E(n)$-coactions on a finite-dimensional algebra $A$. We will first show that any such coaction is in bijective correspondence with an $E(n)^{cop}$-action, which in turn is completely determined by the action of elements $g$ and $x_1$, \ldots, $x_n$ on the elements of $A$.
We will further observe that these correspond to the choice of a tuple $(\varphi, d_1, \ldots, d_n)$ consisting of an involution $\varphi$ and a family $(d_1, \ldots, d_n)$ of $\varphi$-derivations satisfying appropriate conditions (see Thm.~\ref{maingen}).

\bigskip

As already mentioned, the algebra $E(n)$ has a canonical Hopf structure, given by
\[\Delta(g)=g \otimes g, \quad \Delta(x_i)=x_i \otimes g +1 \otimes x_i, \ i=1, \ldots, n\]
\[\varepsilon(g)=1, \quad \varepsilon(x_i)=0, \ i=1, \ldots, n\] 
\[S(g)=g^{-1}=g, \quad S(x_i)=gx_i, \ i=1, \ldots, n.\]
For $P=\lbrace i_1,i_2,\ldots,i_s \rbrace \subseteq \lbrace 1,2, \ldots,n \rbrace$ such that $i_1 < i_2 < \cdots < i_s$, we denote $x_P = x_{i_1}x_{i_2}\cdots x_{i_s}$. If $P=\emptyset$ then $x_{\emptyset} = 1$. The set $\lbrace g^jx_P \ | \ P\subseteq \lbrace 1, \ldots ,n \rbrace, j \in \lbrace 0,1 \rbrace \rbrace$ is a basis of $E(n)$.
Let $F=\lbrace i_{j_1}, i_{j_2}, \ldots, i_{j_r} \rbrace$ be a subset of $P$ and define
\[S(F,P)=(j_1+\cdots+j_r)-\frac{r(r+1)}{2} \textrm{ and } S(\emptyset,P)=0.\]
Then computations show that
\begin{equation}\label{deltagx}
\Delta(g^jx_P)=\sum_{F \subseteq P} (-1)^{S(F,P)} g^j x_F \otimes g^{|F|+j} x_{P \setminus F},
\end{equation}
\begin{equation}\label{antipode}
S(g^jx_P) = (-1)^{j|P|}g^{j+|P|}x_P.
\end{equation}
In \cite{PVO2} it is shown that a $2^{n+1}$-dimensional Clifford algebra $A=Cl(\alpha, \beta_i, \gamma_i, \lambda_{ij})$ admits a canonical $E(n)$-comodule algebra structure $\rho: A \rightarrow A \otimes E(n)$ (that makes it a $E(n)$-cleft extension of $k$) given by
\begin{eqnarray*}
\rho(1_A)&=&1_A \otimes 1_{E(n)}\\
\rho(G) &=&G \otimes g\\
\rho(X_i) &=& X_i \otimes g + 1_A \otimes x_i, \quad i=1, \ldots, n \\
\rho(GX_i) &=& GX_i \otimes 1_{E(n)} +  G \otimes gx_i, \quad i=1, \ldots, n.
\end{eqnarray*}
Our final goal is to understand how to characterize all the possible $E(n)$-comodule algebra structures that a semisimple Clifford algebra $Cl(\alpha, \beta_i, \gamma_i, \lambda_{ij})$ admits. We start by proving that each $E(n)$-coaction on a finite-dimensional algebra $A$ corresponds to a unique $E(n)^{cop}$-action on $A$.

\subsection{From coactions to actions}

In the first place, we want to show that there is an isomorphism of categories 
\[\mathrm{Vec}_k^{E(n)} \cong {_{E(n)^{cop}}}\mathrm{Vec}_k\]
that preserves algebras, so to make clear that each right $E(n)$-coaction corresponds to a unique left $E(n)^{cop}$-action.

We recall that when $H$ is finite-dimensional there is an isomorphism of categories 
\[F':\mathrm{Vec}_k^H \rightarrow {_{H^*}}\mathrm{Vec}_k\]
given by
\[F'(M, \rho)=(M,\mu_\rho),\]
where
\[\mu_{\rho}(h^* \otimes m)=m_0h^*(m_1)\]
for every $h^*\in H^*$ and every $m \in M$ (cf. \cite{CMZ}, p. $10$). The inverse of this functor is given by $G: {_{H^*}}\mathrm{Vec}_k \rightarrow \mathrm{Vec}_k^H$, $G(M, \mu)=(M,\rho_{\mu})$, where
\[\rho_{\mu}(m)=\mu(h_i^* \otimes m) \otimes h_i\]
for every $m \in M$. Here $h_i$ is used to denote a basis of $H$. Both the functor $F'$ and its inverse $G$ send a map to itself, i.e. a map is $H$-colinear if, and only if, is $H^*$-linear. Furthermore, in our case, we also have that $H=E(n)$ is self dual, i.e. there exists an Hopf algebra map $\psi: E(n) \rightarrow E(n)^*$ that gives an isomorphism $E(n) \cong E(n)^*$. It is defined in \cite[Prop .1]{PVO} by 
\[\psi(1)=1^*+g^*=\varepsilon_H, \quad \psi(g)=1^*-g^*, \quad \psi(x_i)=x_i^*+(gx_i)^*, \quad i=1, \ldots, n.\]
We are going to modify $\psi$ in order to define a new Hopf algebra isomorphism $\varphi: E(n)^{cop} \rightarrow E(n)^*$.
\begin{lemma}
The algebra map $\varphi: E(n)^{cop} \rightarrow E(n)^*$ defined by
\[\varphi(1)=1^*+g^*=\varepsilon_{E(n)}, \quad \varphi(g)=1^*-g^*, \quad \varphi(x_i)=-x_i^*+(gx_i)^*, \quad i=1, \ldots, n\]
is an Hopf algebra isomorphism.
\end{lemma}
\begin{proof}
One shows by induction that 
\begin{equation}\label{generalphi}
\varphi(g^jx_P)=(-1)^{\lfloor \frac{|P|+1}{2} \rfloor}(x_P)^*+(-1)^{\lfloor \frac{|P|}{2} \rfloor+j}(gx_P)^*
\end{equation}
for $j \in \lbrace 0,1 \rbrace$ and $P \subseteq \lbrace 1, \ldots n \rbrace$.
Hence we have that
\[(x_P)^*=(-1)^{\lfloor \frac{|P|+1}{2} \rfloor}\varphi \left( \frac{x_P+gx_P}{2}\right), \quad (gx_P)^*=(-1)^{\lfloor \frac{|P|}{2} \rfloor}\varphi \left( \frac{x_P-gx_P}{2}\right),\]
for every $P \subseteq \lbrace 1, \ldots n \rbrace$. This proves that $\varphi$ is a surjective and therefore bijective.
The comultiplication on $E(n)^*$ is given by
\[\Delta_{E(n)^*}((g^jx_P)^*)=\sum_{F \subseteq P} \omega \left[(x_{P \setminus F})^* \otimes(g^jx_F)^* +(-1)^{|P \setminus F|}(gx_{P \setminus F})^* \otimes (g^{j+1}x_F)^*\right],\]
where $\omega=(-1)^{S(F,P)+(|P|+1)|F|+j|P \setminus F|}$.
To show that $\varphi$ is a coalgebra map, one calculates
\[\varepsilon_{E(n)^*}(\varphi(g^jx_P))=\varepsilon_{E(n)^*} \left( (-1)^{\lfloor \frac{|P|+1}{2} \rfloor}(x_P)^*+(-1)^{\lfloor \frac{|P|}{2} \rfloor+j}(gx_P)^*\right)=\delta_{P,\emptyset}=\varepsilon_{E(n)^{cop}}(g^jx_P)\]
and
\small
\begin{eqnarray*}
(\varphi \otimes \varphi)\Delta^{cop}(g^jx_P)&=&\left(\sum_{F \subseteq P} (-1)^{S(F,P)}  \varphi(g^{|F|+j} x_{P \setminus F}) \otimes \varphi(g^j x_F) \right)\\
&=& \sum_{F \subseteq P} (-1)^{S(F,P)} \left[ (-1)^{\lfloor \frac{|P \setminus F|+1}{2} \rfloor}(x_{P \setminus F})^*+(-1)^{\lfloor \frac{|P \setminus F|}{2} \rfloor+|F|+j}(gx_{ P\setminus F})^* \right] \otimes \\
&\otimes & \left[ (-1)^{\lfloor \frac{|F|+1}{2} \rfloor}(x_F)^*+(-1)^{\lfloor \frac{|F|}{2} \rfloor+j}(gx_F)^* \right]\\
&=&\sum_{F \subseteq P} (-1)^{S(F,P)} \left[(-1)^{\lfloor \frac{|P|-|F|+1}{2} \rfloor+\lfloor \frac{|F|+1}{2} \rfloor}(x_{P \setminus F})^* \otimes (x_F)^*+ \right.\\
&+&(-1)^{\lfloor \frac{|P|-|F|}{2} \rfloor+|F|+\lfloor \frac{|F|}{2} \rfloor}(gx_{ P\setminus F})^* \otimes (gx_F)^*+\\
&+&(-1)^{\lfloor \frac{|P|-|F|+1}{2} \rfloor+\lfloor \frac{|F|}{2} \rfloor+j}(x_{P \setminus F})^* \otimes(gx_F)^*+\\
&+& \left.(-1)^{\lfloor \frac{|P|-|F|}{2} \rfloor+|F|+j+\lfloor \frac{|F|+1}{2} \rfloor}(gx_{ P\setminus F})^* \otimes (x_F)^*\right]\\
&\overset{(\star)}{=}&\sum_{F \subseteq P} (-1)^{S(F,P)}\left[ (-1)^{\lfloor \frac{|P|+1}{2} \rfloor+(1+|P|)|F|}(x_{P \setminus F})^* \otimes(x_F)^* \right.+\\
&+& (-1)^{\lfloor \frac{|P|+1}{2} \rfloor+(1+|F|)|P|}(gx_{P \setminus F})^* \otimes (gx_F)^*+ \\
&+& (-1)^{\lfloor \frac{|P|}{2} \rfloor+(1+|F|)|P|+j}(x_{P \setminus F})^* \otimes(gx_F)^* +\\
&+& \left. (-1)^{\lfloor \frac{|P|}{2} \rfloor+(1+|P|)|F|+j}(gx_{P \setminus F})^* \otimes (x_F)^*\right] \\
&=&(-1)^{\lfloor \frac{|P|+1}{2} \rfloor} \Delta((x_P)^*)+(-1)^{\lfloor \frac{|P|}{2} \rfloor+j} \Delta((gx_P)^*)\\
&=& \Delta_{E(n)^*}(\varphi(g^jx_P)).
\end{eqnarray*}
\normalsize
Equality $(\star)$ can be checked case by case, fixing the parity of $|P|$ and $|F|$. For example when both $|P|$ and $|F|$ are even, one has
\begin{eqnarray*}
&&\sum_{F \subseteq P} (-1)^{S(F,P)} \left[(-1)^{\lfloor \frac{|P|-|F|+1}{2} \rfloor+\lfloor \frac{|F|+1}{2} \rfloor}(x_{P \setminus F})^* \otimes (x_F)^*+ \right.\\
&+&(-1)^{\lfloor \frac{|P|-|F|}{2} \rfloor+|F|+\lfloor \frac{|F|}{2} \rfloor}(gx_{ P\setminus F})^* \otimes (gx_F)^*+\\
&+&(-1)^{\lfloor \frac{|P|-|F|+1}{2} \rfloor+\lfloor \frac{|F|}{2} \rfloor+j}(x_{P \setminus F})^* \otimes(gx_F)^*+\\
&+& \left.(-1)^{\lfloor \frac{|P|-|F|}{2} \rfloor+|F|+j+\lfloor \frac{|F|+1}{2} \rfloor}(gx_{ P\setminus F})^* \otimes (x_F)^*\right]\\
&=&\sum_{F \subseteq P} (-1)^{S(F,P)+\frac{|P|}{2}} \left[(x_{P \setminus F})^* \otimes (x_F)^*+ (gx_{ P\setminus F})^* \otimes (gx_F)^*+ \right.\\
&+& \left.(-1)^j(x_{P \setminus F})^* \otimes(gx_F)^*+ (-1)^j(gx_{ P\setminus F})^* \otimes (x_F)^*\right]
\end{eqnarray*}
and
\begin{eqnarray*}
&&\sum_{F \subseteq P} (-1)^{S(F,P)}\left[ (-1)^{\lfloor \frac{|P|+1}{2} \rfloor+(1+|P|)|F|}(x_{P \setminus F})^* \otimes(x_F)^* \right.+\\
&+& (-1)^{\lfloor \frac{|P|+1}{2} \rfloor+(1+|F|)|P|}(gx_{P \setminus F})^* \otimes (gx_F)^*+ \\
&+& (-1)^{\lfloor \frac{|P|}{2} \rfloor+(1+|F|)|P|+j}(x_{P \setminus F})^* \otimes(gx_F)^* +\\
&+& \left. (-1)^{\lfloor \frac{|P|}{2} \rfloor+(1+|P|)|F|+j}(gx_{P \setminus F})^* \otimes (x_F)^*\right] \\
&=&\sum_{F \subseteq P} (-1)^{S(F,P)+\frac{|P|}{2}}\left[ (x_{P \setminus F})^* \otimes(x_F)^* +(gx_{P \setminus F})^* \otimes (gx_F)^*+\right. \\
&+&\left.  (-1)^j(x_{P \setminus F})^* \otimes(gx_F)^*  +(-1)^j(gx_{P \setminus F})^* \otimes (x_F)^*\right] \\
\end{eqnarray*}
\end{proof}
Let again $h_i$ denote the $i$-th element of the basis of $E(n)$ and $h_i^*$ its dual element in $E(n)^*$.
We are going to prove the following proposition.
\begin{proposition}\label{generalduality}
Let $\varphi:E(n)^{cop} \rightarrow E(n)^*$ be the Hopf algebra isomorphism defined by \eqref{generalphi}.
The assignment
\[U:{_{E(n)^{cop}}}\mathrm{Vec}_k \rightarrow\mathrm{Vec}_k^{E(n)}, \quad U(M, \mu)=(M, \rho_{\mu}),\]
where 
\[\rho_{\mu}(m)=\mu(\varphi^{-1}(h_i^*) \otimes m) \otimes h_i\] for every $m \in M$, and $U(f)=f$ for every $E(n)^{cop}$-linear $f$ defines an invertible functor. Its inverse is given by the functor \[V: \mathrm{Vec}_k^{E(n)} \rightarrow{_{E(n)^{cop}}}\mathrm{Vec}_k, \quad V(M, \rho)=(M, \mu_{\rho}),\]
where
\[\mu_{\rho}(h \otimes m)=(\varphi(h)(m_1))m_0\]
for every $h \in E(n) $ and every $m \in M$, and $V(f)=f$ for every $E(n)$-colinear $f$. Moreover both $U$ and $V$ preserve algebras.
\end{proposition}

\begin{proof}
We can use the isomorphism $\varphi$ to define a mapping
\[U:{_{E(n)^{cop}}}\mathrm{Vec}_k \rightarrow\mathrm{Vec}_k^{E(n)}, \quad U(M, \mu)=(M, \rho_{\mu}),\]
where
\[\rho_{\mu}(m)=\mu(\varphi^{-1}(h_i^*) \otimes m) \otimes h_i\]
and $U(f)=f$ for every $E(n)^{cop}$-linear $f$.
To show that $U$ is a functor, we will make use of \cite[Lemma $1$]{CMZ}.
We have, for every $f, g \in E(n)^*$ and $m \in M$,
\begin{eqnarray*}
(\textrm{Id}_M \otimes f \otimes g)(\rho_{\mu} \otimes \textrm{Id}_{E(n)})\rho_{\mu}(m)&=& (\textrm{Id}_M \otimes f \otimes g)(\mu(\varphi^{-1}(h_j'^*) \otimes \mu(\varphi^{-1}(h_i^*) \otimes m)) \otimes h'_j \otimes h_i)\\
&=&\mu(\varphi^{-1}(h_j'^*)\cdot \varphi^{-1}(h_i^*) \otimes m)  f(h'_j) g(h_i)\\
&=&\mu(\varphi^{-1}(h_j'^* \star h_i^*) \otimes m)f(h'_j) g(h_i)\\
&=&\mu(\varphi^{-1}(f(h'_j)h_j'^* \star g(h_i)h_i^*) \otimes m)\\
&=&\mu(\varphi^{-1}(f \star g) \otimes m)\\
&=&\mu(\varphi^{-1}((f \star g)(h_i)h_i^*) \otimes m)\\
&=&\mu(\varphi^{-1}(h_i^*) \otimes m)(f \star g)(h_i)\\
&=&\mu(\varphi^{-1}(h_i^*) \otimes m)f(h_{i_1})g(h_{i_2})\\
&=&(\textrm{Id}_M \otimes f \otimes g)\mu(\varphi^{-1}(h_i^*) \otimes m) \otimes h_{i_1} \otimes h_{i_2}\\
&=&(\textrm{Id}_M \otimes f \otimes g)(\textrm{Id}_M \otimes \Delta)\mu(\varphi^{-1}(h_i^*) \otimes m) \otimes h_i\\
&=&(\textrm{Id}_M \otimes f \otimes g)(\textrm{Id}_M \otimes \Delta)\rho_{\mu}(m).\\
\end{eqnarray*}
By \cite[Lemma $1$]{CMZ}, it follows that $(\rho_{\mu} \otimes \textrm{Id}_{E(n)})\rho_{\mu}=(\textrm{Id}_M \otimes \Delta)\rho_{\mu}$ (cf. proof of Prop. $3$ ibid.).
Furthermore
\begin{eqnarray*}
(\textrm{Id}_M \otimes \varepsilon_{E(n)})\rho_{\mu}(m)&=& \mu(\varphi^{-1}(h_i^*) \otimes m) \otimes \varepsilon_{E(n)}(h_i)\\
&=&\mu(\varphi^{-1}(\varepsilon_{E(n)}(h_i)h_i^*) \otimes m)\\
&=&\mu(\varphi^{-1}(\varepsilon_{E(n)}) \otimes m)\\
&=&\mu(1_{E(n)} \otimes m)\\
&=&m
\end{eqnarray*}
for every $m \in M$ and thus we can conclude that $(M, \rho_{\mu})$ is an object in $\mathrm{Vec}_k^{E(n)}$.
Now let us consider two ${E(n)}^{cop}$-modules $M,N$ and an ${E(n)}^{cop}$-linear map $f:M \rightarrow N$. We prove that $f$ is automatically ${E(n)}$-colinear.
We have
\begin{eqnarray*}
(f \otimes \textrm{Id}_{E(n)})\rho_M(m)&=&(f \otimes \textrm{Id}_{E(n)})(\mu_M(\varphi^{-1}(h_i^*) \otimes m) \otimes h_i)\\
&=& (f \circ \mu_M)(\varphi^{-1}(h_i^*) \otimes m) \otimes h_i\\
&=& \mu_N(\varphi^{-1}(h_i^*) \otimes f(m)) \otimes h_i\\
&=& (\rho_N\circ f)(m).
\end{eqnarray*}
Finally, since $U$ is the identity on morphisms, we can conclude that $U$ is a functor.

$U$ has an inverse $V: \mathrm{Vec}_k^{E(n)} \rightarrow{_{E(n)^{cop}}}\mathrm{Vec}_k$ defined by
\[V: \mathrm{Vec}_k^{E(n)} \rightarrow{_{E(n)^{cop}}}\mathrm{Vec}_k, \quad V(M, \rho)=(M, \mu_{\rho}),\]
where
\[\mu_{\rho}(h \otimes m)=(\varphi(h)(m_1))m_0\]
and $V(f)=f$ for every $E(n)$-colinear $f$. In fact, we have
\begin{eqnarray*}
\mu_{\rho}(\textrm{Id}_{E(n)} \otimes \mu_{\rho})(h' \otimes h \otimes m)&=&\mu_{\rho}(h' \otimes (\varphi(h)(m_1))m_0)\\
&=&(\varphi(h)(m_1))(\varphi(h')(m_{0_1}))m_{0_0}\\
&=&(\varphi(h)(m_{1_2}))(\varphi(h')(m_{1_1}))m_0\\
&=& (\varphi(h') \star \varphi(h))(m_1)m_0\\
&=&\varphi(h'h)(m_1)m_0\\
&=&\mu_{\rho}(m \otimes \textrm{Id}_M)(h' \otimes h \otimes m)
\end{eqnarray*}
for every $h, h' \in E(n)$ and every $m \in M$. In the third equality one must pay attention to the fact that $M$ is an $E(n)^{cop}$-module.
Next
\begin{eqnarray*}
\mu_{\rho}(1_{E(n)} \otimes m)&=&(\varphi(1_{E(n)})(m_1))(m_0)\\
&=& \varepsilon(m_1)m_0\\
&=&m
\end{eqnarray*}
for every $m \in M$ and so we can conclude that $(M, \mu_{\rho})$ is an object in ${_{E(n)^{cop}}}\mathrm{Vec}_k$. Then we consider two $E(n)$-comodules $M,N$ and an $E(n)$-colinear map $f:M \rightarrow N$. We prove that $f$ is automatically $E(n)^{cop}$-linear.
We have
\begin{eqnarray*}
(f \circ \mu_M)(h \otimes m)&=&(\varphi(h)(m_1))f(m_0)\\
&=&(\varphi(h)(f(m)_1)f(m)_0\\
&=&\mu_N(h \otimes f(m))\\
&=& \mu_N(\textrm{Id}_{E(n)} \otimes f)(h \otimes m)
\end{eqnarray*}
for every $h \in E(n)$ and every $m \in M$. Since $V$ is the identity on morphisms, we conclude that $V$ is a functor.

We also want to prove that $UV=\textrm{Id}_{\mathrm{Vec}_k^{E(n)}}$ and $VU=\textrm{Id}_{{_{E(n)^{cop}}}\mathrm{Vec}_k}$.
We have $UV(M, \rho)=(M,\rho_{\mu_{\rho}})$
and
\begin{eqnarray*}
\rho_{\mu_{\rho}}(m)&=&\mu_{\rho}(\varphi^{-1}(h_i^*) \otimes m) \otimes h_i\\
&=&(\varphi\varphi^{-1}(h_i^*))(m_1)m_0 \otimes h_i\\
&=&h_i^*(m_1)m_0 \otimes h_i\\
&=&m_0 \otimes h_i^*(m_1) h_i\\
&=& m_0 \otimes m_1\\
&=&\rho(m)
\end{eqnarray*}
for every $m \in M$. Moreover $VU(M, \mu)=(M, \mu_{\rho_{\mu}})$ and
\begin{eqnarray*}
\mu_{\rho_{\mu}}(h \otimes m) &=& (\varphi(h)(h_i))\mu(\varphi^{-1}(h_i^*) \otimes m)\\
&=&\mu([\varphi(h)(h_i)]\varphi^{-1}(h_i^*) \otimes m)\\
&=&\mu(\varphi^{-1}([\varphi(h)(h_i)]h_i^*) \otimes m).\\
\end{eqnarray*}
Notice that the $\varphi(h)(h_i)$'s are the coordinates of the vector $\varphi(h) \in E(n)^*$ on the dual basis $h_i^*$, which means that $[\varphi(h)(h_i)]h_i^*=\varphi(h)$. Hence
\[\mu_{\rho_{\mu}}(h \otimes m)=\mu(\varphi^{-1}(\varphi(h)) \otimes m)=\mu(h \otimes m)\]
for any $h \in E(n)$ and any $m \in M$. Therefore we have proved that the functors $U$ and $V$ previously defined give an isomorphism of categories
\[{_{E(n)^{cop}}}\mathrm{Vec}_k\cong \mathrm{Vec}_k^{E(n)}.\]

Finally, the functors $U$ and $V$ send algebras to algebras, i.e. the image $(A, \rho_{\mu})$ of an $E(n)^{cop}$-module algebra $(A, \mu)$ is an $E(n)$-comodule algebra and, conversely, the image $(V(A), \mu_{\rho})$ of an $E(n)$-comodule algebra $(A, \rho)$ is an $E(n)^{cop}$-module algebra. Suppose $(A, \mu)$ is an $E(n)^{cop}$-module algebra. This means that
\eqref{modalg1} and \eqref{modalg2} hold. To prove that $A$ is also an $E(n)$-comodule algebra we need to show \eqref{comodalg2} and \eqref{comodalg3} are satisfied.
For any $f \in E(n)^*$ we have
\begin{eqnarray*}
(\textrm{Id}_A \otimes f)(\rho_{\mu}(a)\rho_{\mu}(b))&=&(\textrm{Id}_A \otimes f)([\mu(\varphi^{-1}(h_i^*) \otimes a) \otimes h_i][\mu(\varphi^{-1}(h_j'^*) \otimes b) \otimes h_j'])\\
&=&\mu(\varphi^{-1}(h_i^*) \otimes a)\mu(\varphi^{-1}(h_j'^*) \otimes b)f(h_ih_j')\\
&=&\mu(\varphi^{-1}(h_i^*) \otimes a)\mu(\varphi^{-1}(h_j'^*) \otimes b)f_1(h_i)f_2(h_j')\\
&=&\mu(\varphi^{-1}(f_1(h_i)h_i^*) \otimes a)\mu(\varphi^{-1}(f_2(h_j')h_j'^*) \otimes b)\\
&=&\mu(\varphi^{-1}(f_1) \otimes a)\mu(\varphi^{-1}(f_2) \otimes b)\\
&=&\mu(\varphi^{-1}(f)_2 \otimes a)\mu(\varphi^{-1}(f)_1 \otimes b)\\
&\overset{\eqref{modalg1}}{=}&\mu(\varphi^{-1}(f) \otimes ab)\\
&=&\mu(\varphi^{-1}(f(h_i)h_i^*) \otimes ab)\\
&=&\mu(\varphi^{-1}(h_i^*) \otimes ab)f(h_i)\\
&=&(\textrm{Id}_A \otimes f)(\mu(\varphi^{-1}(h_i^*) \otimes ab) \otimes h_i)\\
&=&(\textrm{Id}_A \otimes f)\rho_{\mu}(ab).
\end{eqnarray*}
By \cite[Lemma $1$]{CMZ}, it follows that $\rho_{\mu}(a)\rho_{\mu}(b)=\rho_{\mu}(ab)$ for any $a,b \in A$ and thus \eqref{comodalg2} holds true.

Furthermore
\begin{eqnarray*}
\rho_{\mu}(1_A)&=&\mu(\varphi^{-1}(h_i^*) \otimes 1_A) \otimes h_i\\
&\overset{\eqref{modalg2}}{=}& \varepsilon_{E(n)}(\varphi^{-1}(h_i^*))1_A \otimes h_i\\
&=& \varepsilon_{E(n)^*}(h_i^*)1_A \otimes h_i\\
&=& h_i^*(1_H)1_A\otimes h_i\\
&=&1_A \otimes 1_{E(n)}
\end{eqnarray*}
and thus \eqref{comodalg3} is satisfied and $(A, \rho_{\mu})$ is an $E(n)$-comodule algebra.

Now suppose $(A, \rho)$ is an $E(n)$-comodule algebra. This means that \eqref{comodalg2} and \eqref{comodalg3} hold. To prove that $(A, \mu_{\rho})$ is an $E(n)^{cop}$-module algebra we need to show that \eqref{modalg1} and \eqref{modalg2} are verified. We have
\begin{eqnarray*}
\mu_{\rho}(h \otimes ab)&=&(\varphi(h)((ab)_1))(ab)_0\\
&\overset{\eqref{comodalg2}}{=}&(\varphi(h)(a_1b_1))a_0b_0\\
&=&[(\varphi(h))_1(a_1)(\varphi(h))_2(b_1)]a_0b_0\\
&=&[\varphi(h_2)(a_1)\varphi(h_1)(b_1)]a_0b_0\\
&=&(\varphi(h_2)(a_1)a_0)(\varphi(h_1)(b_1)b_0)\\
&=&\mu_{\rho}(h_2 \otimes a)\mu_{\rho}(h_1 \otimes b)
\end{eqnarray*}
for every $h \in E(n)$ and every $a,b \in A$, thus \eqref{modalg1} is satisfied.
Then
\begin{eqnarray*}
\mu_{\rho}(h \otimes 1_A)&\overset{\eqref{comodalg3}}{=}&(\varphi(h)(1_{E(n)}))1_A\\
&=& \varepsilon_{E(n)^*}(\varphi(h))1_A\\
&=& \varepsilon_{E(n)}(h)1_A,
\end{eqnarray*}
for every $h \in E(n)$ and we can conclude.
\end{proof}

The last proposition proves that each $E(n)$-coaction $\rho$ on a finite dimensional algebra $A$ can be expressed in terms of a unique $E(n)^{cop}$-action:
\[\rho(a)=\mu(\varphi^{-1}(h_i^*) \otimes a) \otimes h_i \quad \textrm{for every } a \in A,\]
and explicitly
\begin{eqnarray*}
\rho(a)&=&\sum_{\underset{P}{j=0,1}}\mu(\varphi^{-1}((g^jx_P)^*) \otimes a) \otimes g^jx_P\\
&=&\sum_{\underset{P}{j=0,1}}\mu\left(\varphi^{-1}\left((-1)^{\lfloor \frac{|P|+1-j}{2} \rfloor}\varphi \left( \frac{x_P +(-1)^jgx_P}{2}\right)\right) \otimes a \right) \otimes g^jx_P\\
&=&\sum_{\underset{P}{j=0,1}}(-1)^{\lfloor \frac{|P|+1-j}{2} \rfloor}\mu\left( \frac{x_P +(-1)^jgx_P}{2} \otimes a \right) \otimes g^jx_P\\
&=&\sum_P(-1)^{\lfloor \frac{|P|+1}{2} \rfloor}\mu\left( \frac{x_P +gx_P}{2} \otimes a \right) \otimes x_P+(-1)^{\lfloor \frac{|P|}{2} \rfloor}\mu\left( \frac{x_P -gx_P}{2} \otimes a \right) \otimes gx_P\\\
&=&\sum_P(-1)^{\lfloor \frac{|P|+1}{2} \rfloor} \left[ \mu(x_P \otimes a) \otimes \frac{x_P +(-1)^{|P|}gx_P}{2}+\mu(gx_P \otimes a) \otimes \frac{x_P +(-1)^{|P|+1}gx_P}{2}\right],\\
\end{eqnarray*}
for every $a \in A$, where $P$ is taken to run over all ordered subsets $\lbrace i_1<i_2< \cdots <i_s \rbrace$ of $\lbrace 1, \ldots, n \rbrace$.
It is clear that, since $\mu(gx_P \otimes a)=\mu(g \otimes \mu(x_P \otimes a))$, and $\mu(x_P \otimes a)=\mu(x_{i_1} \otimes \mu(x_{i_2} \otimes \ldots \mu(x_{i_s} \otimes a)))$, then $\rho$ is completely determined once we know how $g$ and each $x_i$ act on the elements of $A$. As a matter of fact $E(n)^{cop}$-actions on a finite-dimensional algebra $A$ are in bijective correspondence with $n+1$-uples $(\varphi, d_1, \ldots, d_n)$ of suitable maps as we will prove in the sequel.

\subsection{From actions to involutions and derivations}

We have seen that each $E(n)^{cop}$-action on a finite dimensional algebra $A$ is completely determined by the choice of an action of $g, x_1, \ldots, x_n$ on the elements of $A$. This, in turn, is equivalent to the choice of an $n+1$-uple $(\varphi, d_1, \ldots, d_n)$, where $\varphi$ is an involution of $A$, the $d_i$'s are $\varphi$-derivations on $A$ and $d_i^2 \equiv 0$, $\varphi d_i=-d_i \varphi$, $d_id_j=-d_jd_i$  for every $i,j=1, \ldots, n$.  
We will prove this claim by making use of a well-established approach (see \cite{CY,ms} and Examples $4.1.6$ and $4.1.8$ in \cite{mo}).

\begin{proposition}\label{firstprop2}
Let $A$ be a finite-dimensional algebra over a field $k$ of characteristic $\chara k \neq 2$. Then an $E(n)^{cop}$-action on $A$ is completely determined by the choice of:
\begin{enumerate}[(i)]
\item an automorphism $\varphi$ of $A$ of order $o(\varphi) \leq 2$ (i.e. an involution of $A$),\label{item1}
\item a family $(d_1, \ldots, d_n)$ of $\varphi$-derivations such that $d_i^2 \equiv 0$, $\varphi d_i = -d_i \varphi$ and $d_id_j=-d_jd_i$ for $i,j=1, \ldots, n$.\label{item2}
\end{enumerate}
This correspondence is bijective.
\end{proposition}

\begin{proof}
Given an $E(n)^{cop}$-action $\mu:E(n) \otimes A \rightarrow A$, set $\varphi(a):=\mu(g \otimes a)$. Clearly $\varphi: A \rightarrow A$ is a $k$-linear map and furthermore, for every $a, b \in A$,
\begin{eqnarray*}
\varphi(ab)=\mu(g \otimes ab)\overset{\eqref{modalg1}}{=}\mu(g_2 \otimes a)\mu(g_1 \otimes b)=\mu(g \otimes a)\mu(g \otimes b)=\varphi(a)\varphi(b),\\
\varphi(1_A)=\mu(g \otimes 1_A)\overset{\eqref{modalg2}}{=}\varepsilon(g)1_A=1_A,
\end{eqnarray*}
which means that $\varphi$ is an algebra endomorphism. Then
\begin{eqnarray*}
\varphi(\varphi(a))=\mu(g \otimes \mu(g \otimes a))=\mu(g^2 \otimes a)=\mu(1_H \otimes a)=a,
\end{eqnarray*}
for every $a \in A$, so $o(\varphi) \leq 2$. This means that $\varphi$ is an involution of $A$.
Next, for any $i=1, \ldots, n$, if we set $d_i(a):=\mu(x_i \otimes a)$, we have a $k$-linear map $d_i:A \rightarrow A$ such that
\begin{eqnarray*}
d_i(ab)=\mu(x_i \otimes ab)\overset{\eqref{modalg1}}{=}\mu(x_i \otimes a)\mu(1_H \otimes b)+\mu(g \otimes a)\mu(x_i \otimes b)=d_i(a)b+\varphi(a)d_i(b),
\end{eqnarray*}
for every $a,b \in A$, that is, $d_i:A \rightarrow A$ is a $\varphi$-derivation. Furthermore, for any $a \in A$,
\begin{eqnarray*}
d_i(d_i(a))=\mu(x_i \otimes \mu(x_i \otimes a))=\mu(x_i^2 \otimes a)=\mu(0 \otimes a)=0,
\end{eqnarray*}
which means that $d_i^2$ is the zero map.
Finally,
\begin{eqnarray*}
d_i(\varphi(a))=\mu(x_i \otimes \mu(g \otimes a))=\mu(x_i g \otimes a)=-\mu(gx_i \otimes a)=-\mu(g \otimes \mu(x_i \otimes a))=-\varphi(d_i(a)),
\end{eqnarray*}
for every $a \in A$, which means that $\varphi$ and $d_i$ must anticommute. 
Similarly, one can easily check that $d_id_j=d_jd_i$ for every $i,j=1, \ldots, n$. 

In this way we have established an assignment
\[\Phi: \lbrace E(n)^{cop}\textrm{-actions on a f.d. algebra}  \rbrace \rightarrow \lbrace (\varphi, d_1, \ldots, d_n) \ | \ \textrm{satisfying } \eqref{item1}, \eqref{item2} \rbrace\]
\[(\mu: E(n)^{cop} \otimes A \rightarrow A) \overset{\Phi}{\longmapsto }(\varphi:=\mu(g \otimes \--), d_i:=\mu(x_i \otimes \--)).\]

\medskip

Conversely, let us fix an $n+1$-uple of $k$-linear maps $(\varphi:A \rightarrow A, d_i: A \rightarrow A)$, $i=1, \ldots, n$, satisfying \eqref{item1}, \eqref{item2}. We are going to show that we can build an $E(n)^{cop}$-action on $A$. 
We define a $k$-linear map $\mu:E(n)^{cop} \otimes A \rightarrow A$, by  setting 
\begin{equation}\label{mu1}
\mu(g^jx_P \otimes a):=\varphi^j(d_{i_1}d_{i_2} \cdots d_{i_s}(a)), \quad \mu(g^jx_{\emptyset} \otimes a)=\mu(g^j \otimes a):=\varphi^j(a)
\end{equation}
for every $j \in \lbrace 0,1 \rbrace$, $P \subseteq \lbrace 1, \ldots, n\rbrace$ and $a \in A$. Here $x_P=x_{i_1}x_{i_2}\cdots x_{i_s}$, when $P=\lbrace i_1 <i_2< \ldots <i_s \rbrace$, and $x_{\emptyset}=1_{E(n)}$  as usual. We need to show that $(A, \mu)$ is an $E(n)^{cop}$-module and that \eqref{modalg1}-\eqref{modalg2} hold. 

By definition $\mu(1_{E(n)} \otimes a)=a$, therefore we only need to prove that 
\begin{equation}\label{module1}
\mu(h' \otimes \mu(h \otimes a))=\mu(h'h \otimes a)
\end{equation}
for every $h, h' \in E(n)^{cop}$ and every $a \in A$, to show that $(A, \mu)$ is an $E(n)^{cop}$-module. It is sufficient to pick $h$ and $h'$ among elements of the basis of $E(n)$. Let $x_P=x_{i_1}x_{i_2}\cdots x_{i_s}$ and $x_Q=x_{j_1}x_{j_2}\cdots x_{j_t}$. We have
\begin{eqnarray*}
\mu(g^kx_Q \otimes \mu(g^jx_P \otimes a))&\overset{\eqref{mu1}}{=}&\mu(g^kx_Q \otimes \varphi^j(d_{i_1}d_{i_2} \cdots d_{i_s}(a)))\\
&\overset{\eqref{mu1}}{=}&\varphi^k(d_{j_1}d_{j_2} \cdots d_{j_t}(\varphi^j(d_{i_1}d_{i_2} \cdots d_{i_s}(a))).
\end{eqnarray*}
Since $\varphi d_i=-d_i \varphi$ for every $i=1, \ldots, n$, it follows that
\[\varphi^k(d_{j_1}d_{j_2} \cdots d_{j_t}(\varphi^j(d_{i_1}d_{i_2} \cdots d_{i_s}(a)))=(-1)^{jt}\varphi^{j+k}(d_{j_1}d_{j_2} \cdots d_{j_t}(d_{i_1}d_{i_2} \cdots d_{i_s}(a))).\]
Suppose there are $l,m \in \mathbb{N}$ such that $j_l=i_m$. Since the $d_i$'s anticommute and $d_i^2 \equiv 0$, by changing sign accordingly, one can bring $d_{j_l}$ and $d_{i_m}$ together and show that the whole term must vanish.
In this case, it is clear that $P \cap Q \neq \emptyset$, and $g^jx_Qg^jx_P=0$, since $x_{j_l}x_{i_m}=x_{i_m}^2=0$.
Therefore
\[\mu(g^kx_Q \otimes \mu(g^jx_P \otimes a))=0=\mu(0 \otimes a)=\mu(g^jx_Qg^jx_P \otimes a),\]
i.e. \eqref{module1} holds.
Now suppose that $P \cap Q = \emptyset$. Then we have
\[\mu(g^kx_Q \otimes \mu(g^jx_P \otimes a))=(-1)^{jt}\varphi^{j+k}(d_{j_1}d_{j_2} \cdots d_{j_t}d_{i_1}d_{i_2} \cdots d_{i_s}(a)),\]
where all the involved derivations are different. Let $\alpha_{Q,P} \in \lbrace -1, 1 \rbrace$ be defined by $x_Qx_P=\alpha_{Q,P}x_{Q \cup P}$. If we denote by $d_R$ the ordered composition of derivations indexed by the ordered set $R$ (e.g. $d_P=d_{i_1}d_{i_2} \cdots d_{i_s}$, where $P=\lbrace i_1 <i_2 < \ldots < i_s \rbrace$), it is clear that also $d_Qd_P=\alpha_{Q,P}d_{Q \cup P}$. We can conclude that
\begin{eqnarray*}
\mu(g^kx_Q \otimes \mu(g^jx_P \otimes a))&=&(-1)^{jt}\varphi^{j+k}(d_{j_1}d_{j_2} \cdots d_{j_t}d_{i_1}d_{i_2} \cdots d_{i_s}(a))\\
&=&(-1)^{jt}\varphi^{j+k}(d_Qd_P(a))\\
&=&(-1)^{jt}\alpha_{Q,P}\varphi^{j+k}(d_{Q \cup P}(a))\\
&\overset{\eqref{mu1}}{=}&(-1)^{js}\alpha_{Q,P} \mu(g^{j+k}x_{Q \cup P} \otimes a)\\
&=&(-1)^{jt}\mu(g^{j+k}x_Q x_P \otimes a)\\
&=&\mu(g^kx_Q g^jx_P \otimes a)\\
\end{eqnarray*}
Therefore $(A, \mu)$ is an $E(n)^{cop}$-module. Now we prove that it is an $E(n)^{cop}$-module algebra, i.e. that \eqref{modalg1} and \eqref{modalg2} hold.

Let $d_{\emptyset}:=\textrm{Id}_A$. We have
\[\mu(g^jx_P \otimes 1_A)=\varphi^j(d_P(1_A))=\delta_{P,\emptyset}=\varepsilon(g^jx_P)1_A,\]
therefore \eqref{modalg2} holds true.
Next
\begin{eqnarray*}
\mu(x_i \otimes ab) &=& d_i(ab)\\
&=& d_i(a)b+\varphi(a)d_i(b)\\
&=&\mu(x_i \otimes a)\mu(1_{E(n)} \otimes b)+ \mu(g \otimes a)\mu(x_i \otimes b)\\
&=&\mu((x_i)_2 \otimes a)\mu((x_i)_1 \otimes b)
\end{eqnarray*}
for every $a,b \in A$ and every $i=1, \ldots, n$.
Suppose that $\mu(x_P \otimes ab)=\mu((x_P)_2 \otimes a)\mu((x_P)_1 \otimes b)$ holds true for every $a, b \in A$ and every $P$ with cardinality $m \geq 1$.
Take $Q=\lbrace i_1, i_2, \ldots i_{m+1} \rbrace$ (where all elements are different, so $|Q|=m+1$) and let $Q'=Q \setminus \lbrace i_1 \rbrace=\lbrace i_2, \ldots, i_{m+1} \rbrace$. Then
\begin{eqnarray*}
\mu(x_Q \otimes ab)&=&\mu(x_{i_1} \otimes \mu(x_{Q'} \otimes ab))\\
&\overset{ind. hyp.}{=}&\mu(x_{i_1} \otimes \mu((x_{Q'})_2 \otimes a)\mu((x_{Q'})_1 \otimes b))\\
&=&\mu((x_{i_1})_2 \otimes \mu((x_{Q'})_2 \otimes a))\mu((x_{i_1})_1 \otimes\mu((x_{Q'})_1 \otimes b))\\
&=&\mu((x_{i_1})_2(x_{Q'})_2 \otimes a))\mu((x_{i_1})_1(x_{Q'})_1 \otimes b))\\
&=&\mu((x_Q)_2 \otimes a))\mu((x_Q)_1 \otimes b)).\\
\end{eqnarray*}
Thus we have proved $\mu(x_P \otimes ab)=\mu((x_P)_2 \otimes a))\mu((x_P)_1 \otimes b))$ by induction on $m=|P|$.
Clearly $\mu(1_{E(n)} \otimes ab)=ab=\mu(1_{E(n)} \otimes a)\mu(1_{E(n)} \otimes b)$ and $\mu(g \otimes ab)=\varphi(ab)=\varphi(a)\varphi(b)=\mu(g \otimes a)\mu(g \otimes b)$ for every $a, b \in A$.
Finally
\begin{eqnarray*}
\mu(gx_P \otimes ab)&=&\mu(g \otimes \mu(x_P \otimes ab))\\
&=&\mu(g \otimes \mu((x_P)_2 \otimes a)\mu((x_P)_1 \otimes b))\\
&=&\varphi(\mu((x_P)_2 \otimes a)\mu((x_P)_1 \otimes b))\\
&=&\varphi(\mu((x_P)_2 \otimes a))\varphi(\mu((x_P)_1 \otimes b))\\
&=&\mu(g \otimes \mu((x_P)_2 \otimes a))\mu(g \otimes \mu((x_P)_1 \otimes b))\\
&=&\mu(g(x_P)_2 \otimes a))\mu(g(x_P)_1 \otimes b))\\
&=&\mu((gx_P)_2 \otimes a))\mu((gx_P)_1 \otimes b))\\
\end{eqnarray*}
for every $a,b \in A$. We have proved that \eqref{modalg1} holds for every $a, b \in A$ and every $h$ of the canonical basis of $E(n)$. Since the involved maps are $k$-linear, this implies that \eqref{modalg1} holds for every $h \in E(n)$ and every $a, b \in A$. 

In this way we have established an assignment
\[\Psi: \lbrace (\varphi, d_1, \ldots, d_n) \ | \ \textrm{satisfying } \eqref{item1}, \eqref{item2} \rbrace \rightarrow \lbrace E(n)^{cop}\textrm{-actions on a f.d. algebra}  \rbrace \]
\begin{align*}
(\varphi, d_1, \ldots, d_n) \quad  \overset{\Psi} {\longmapsto}  \quad \mu:E(n)^{cop} \otimes A & \longrightarrow A\\
g^jx_P \otimes a & \longmapsto \varphi^j(d_P(a)).\\
\end{align*}
It is straightforward to check that the assignments $\Phi$ and $\Psi$ are inverse to each other, and therefore that the correspondence between $E(n)^{cop}$-actions and $n+1$-uples $(\varphi,d_1, \ldots, d_n)$ satisfying \eqref{item1}, \eqref{item2} is bijective.
\end{proof}

\subsection{The explicit correspondence}

Finally we can employ the correspondence established in Proposition~\ref{firstprop2} to write down the explicit expression of a coaction $\rho:A \rightarrow A \otimes E(n)$ in terms of the associated involution $\varphi$ and derivations $d_1, \ldots, d_n$. We have seen that every $E(n)$-coaction $\rho$ on a finite-dimensional algebra $A$ is defined by
\begin{equation}\label{explgen}
\rho(a)= \underset{P \subseteq \lbrace 1, \ldots, n \rbrace}{\sum}(-1)^{\lfloor \frac{|P|+1}{2} \rfloor} \left[ x_P \cdot a \otimes \frac{x_P +(-1)^{|P|}gx_P}{2}+gx_P \cdot a \otimes \frac{x_P +(-1)^{|P|+1}gx_P}{2}\right].
\end{equation}
for every $a \in A$, where $\cdot$ denotes a (left) $E(n)^{cop}$-action. Since, by Proposition~\ref{firstprop2}, each $E(n)^{cop}$-action is in bijective correspondence with an $n+1$-uple $(\varphi, d_1, \ldots, d_n)$ where $\varphi:A \rightarrow A$ is an involution and each $d_i:A \rightarrow A$ is a $\varphi$-derivation such that $d_i^2 \equiv 0$, $\varphi d_i =-d_i \varphi$ and $d_id_j=-d_jd_i$ for every $i,j=1, \ldots, n$, \eqref{explgen} rewrites as
\[\rho(a)= \underset{P \subseteq \lbrace 1, \ldots, n \rbrace}{\sum}(-1)^{\lfloor \frac{|P|+1}{2} \rfloor} \left[ d_P(a) \otimes \frac{x_P +(-1)^{|P|}gx_P}{2}+\varphi(d_P(a)) \otimes \frac{x_P +(-1)^{|P|+1}gx_P}{2}\right],\]
where again $d_P=d_{i_1}d_{i_2}\cdots d_{i_s}$ for $P=\lbrace i_1<i_2<\ldots<i_s \rbrace$ and $d_{\emptyset}=\textrm{Id}_A$.
\begin{theorem}\label{maingen}
Let $A$ be a finite dimensional algebra over a field $k$ of characteristic $\chara k \neq 2$. Then an $E(n)$-comodule algebra structure on $A$ is given by:
\begin{equation}\label{coactionfinal}
\rho(a)= \underset{P \subseteq \lbrace 1, \ldots, n \rbrace}{\sum}(-1)^{\lfloor \frac{|P|+1}{2} \rfloor} \left[ d_P(a) \otimes \frac{x_P +(-1)^{|P|}gx_P}{2}+\varphi(d_P(a)) \otimes \frac{x_P +(-1)^{|P|+1}gx_P}{2} \right],
\end{equation}
where
\begin{enumerate}
\item $\varphi$ is an automorphism  of $A$ of order $o(\varphi) \leq 2$ (i.e. an involution of $A$),
\item $d_{\emptyset}=\textrm{Id}_A$ and $d_P=d_{i_1}d_{i_2}\cdots d_{i_{|P|}}$ is a composition of $\varphi$-derivations such that $d_i^2 \equiv 0$, $\varphi d_i = -d_i \varphi$ and $d_id_j=-d_jd_i$.
\end{enumerate}
\end{theorem}

It is now clear that to have a classification of all $E(n)$-coactions on a finite-dimensional algebra $A$ is equivalent to have a classification of all the involutions of $A$ and the corresponding skew-derivations satisfying the hypothesis of Theorem~\ref{maingen}.

\subsection{The space of coinvariants}

We recall that, given a Hopf algebra $H$ and an $H$-comodule algebra $A$ with structure map $\rho: A \rightarrow A \otimes H$, the space of coinvariants $A^{co H}$ is defined as
\[A^{co H}=\lbrace a \in A \ | \ \rho(a)=a \otimes 1_H \rbrace.\]

\begin{proposition}
Let $A$ be a finite dimensional algebra over a field $k$ of characteristic $\chara k \neq 2$ and let $\rho:A \rightarrow A \otimes E(n)$ be a coaction of $E(n)$ on $A$. Let $(\varphi, d_1, \ldots, d_n)$ be the tuple mentioned in \autoref{maingen} associated to $\rho$. Then
\[A^{coH}=\ker (\varphi - \textrm{Id}_A) \cap \bigcap_{i=1}^n \ker d_i.\]
\end{proposition}

\begin{proof}
$(\subseteq)$ Suppose $a \in A^{coH}$. Then $\rho(a)=a \otimes 1$ and by \eqref{coactionfinal} we find
\[a \otimes 1 =\underset{P \subseteq \lbrace 1, \ldots, n \rbrace}{\sum}(-1)^{\lfloor \frac{|P|+1}{2} \rfloor} \left[ d_P(a) \otimes \frac{x_P +(-1)^{|P|}gx_P}{2}+\varphi(d_P(a)) \otimes \frac{x_P +(-1)^{|P|+1}gx_P}{2}\right].\]
By linear independence we must have
\[ d_i(a) \otimes \frac{x_i -gx_i}{2}=0\]
for every $i=1, \ldots, n$ and therefore $a \in \ker d_i$ for every $i=1, \ldots, n$.
As a straightforward consequence we find
\[a \otimes 1 = \rho(a)= a \otimes \frac{1 +g}{2}+\varphi(a) \otimes \frac{1 -g}{2}.\]
This forces $\varphi(a)=a$ and therefore $a \in \ker (\varphi - \textrm{Id}_A)$.

$(\supseteq)$ Suppose $a \in \ker (\varphi - \textrm{Id}_A) \cap \bigcap_{i=1}^n \ker d_i$. Then it is immediate to check that
$\rho(a)=a \otimes 1$ using \eqref{coactionfinal}.
\end{proof}

\section{The case of semisimple Clifford algebras}\label{gensemisimplecase}

Our aim is to determine all $E(n)$-coactions on a semisimple Clifford algebra $A=Cl(\alpha, \beta_i, \gamma_i, \lambda_{ij})$, that is, thanks to \autoref{maingen}, all the involutions and the skew-derivations of $A$. Since, by \autoref{class1}, a semisimple Clifford algebra is either a central simple algebra or the product of two isomorphic central simple algebras,  it turns out that the maps we are looking for must be of a very specific type. As \autoref{class1} suggests, the exact form of these maps will depend on the parity of $n$ and thus we will discuss the two cases separately. Let us first remind a couple of useful definitions.

\begin{definition}
Let $A$ be a $k$-algebra and $\varphi:A \rightarrow A$ an automorphism of $A$.
$\varphi$ is said to be an \emph{inner} automorphism if there exists an invertible element $c \in A$ such that
\[\varphi(a)=c^{-1}ac \qquad \textrm{for all } a \in A.\]
A $\varphi$-derivation $d: A \rightarrow A$ is said to be \emph{inner} if there exists an element $u \in A$ such that
\[d(a)=ua-\varphi(a)u  \qquad \textrm{for all } a \in A.\]
\end{definition}
We will write $\varphi_c:A \rightarrow A$ to denote the inner morphism associated to the element $c$, i.e. such that $\varphi_c(a)=c^{-1}ac$ for all $a \in A$, and $d_i:A \rightarrow A$ to indicate the inner $\varphi$-derivation associated to the element $u_i$, i.e. such that $d_i(a)=u_ia-\varphi(a)u_i$ for every $a \in A$. From now on we will always assume $\chara k \neq 2$ even when not stated explicitly.

\subsection{The case \texorpdfstring{$n$}{n} odd}

Let $A=Cl(\alpha, \beta_i, \gamma_i, \lambda_{ij})$ be a semisimple Clifford algebra of dimension $2^{n+1}$ with $n$ odd. \autoref{class1} tells us that $A$ is a central simple algebra and the Skolem-Noether Theorem \cite[Thm. IV.1.8]{La} ensures that every automorphism of $A$ \--- and thus every involution of $A$ \--- is inner. It can be proved that the simplicity of $A$ also forces every skew-derivation $d$ to be inner, but in our case, since we are interested in skew-derivations anticommuting with involutions, we will see that the fact that $d$ is inner comes as an immediate consequence of this request.

\begin{theorem}\label{oddcase}
Let $A=Cl(\alpha, \beta_i, \gamma_i, \lambda_{ij})$ be a $2^{n+1}$-dimensional Clifford algebra with $n$ odd. Assume that $A$ is semisimple. Then an $E(n)$-coaction on $A$ is given by:
\[\rho(a)= \underset{P \subseteq \lbrace 1, \ldots, n \rbrace}{\sum}(-1)^{\lfloor \frac{|P|+1}{2} \rfloor} \left[ d_P(a) \otimes \frac{x_P +(-1)^{|P|}gx_P}{2}+c^{-1}d_P(a)c \otimes \frac{x_P +(-1)^{|P|+1}gx_P}{2} \right]\]
where
\begin{enumerate}
\item $c$ is an invertible element of $A$ such that $c^2 \in k$,
\item $d_{\emptyset}=\textrm{Id}_A$ and $d_P=d_{i_1}d_{i_2}\cdots d_{i_{|P|}}$ is a composition of maps $d_i:A \rightarrow A$ defined by $d_i(a)=u_ia-c^{-1}acu_i$ for every $a \in A$, such that $u_ic+cu_i=0$, $u_i^2 \in k$ and $u_iu_j+u_ju_i \in k$.
\end{enumerate}
\end{theorem}

\begin{proof}
Let $\rho:A \rightarrow A \otimes E(n)$ be an $E(n)$-coaction on $A$. Then, by \autoref{maingen}, $\rho$ is given by \eqref{coactionfinal}, where $(\varphi, d_1, \ldots, d_n)$ satisfy
\begin{eqnarray}
&\varphi^2 =\textrm{Id}_A,\label{varphi^2}\\
&d_i^2 \equiv 0 \ \textrm{for every } i=1, \ldots, n, \label{d^2}\\
&\varphi d_i=-d_i\varphi \ \textrm{for every } i=1, \ldots, n,\label{varphid}\\
&d_id_j=-d_jd_i  \ \textrm{for every } i,j=1, \ldots, n\label{didj}.
\end{eqnarray}
As already observed, $A$ is a central simple algebra by \autoref{class1} and thus $\varphi$ must be an inner involution. Let us show what conditions \eqref{varphi^2}-\eqref{didj} translate into when $\varphi=\varphi_c$ is inner.
\[\varphi_c^2 \equiv \textrm{Id} \iff c^{-2}ac^2=a \quad \textrm{for all } a \in A \iff ac^2=c^2a \quad \textrm{for all } a \in A \iff c^2 \in \mathcal{Z}(A)=k.\]
Next, we want to prove that a $\varphi_c$-derivations $d$ anticommute with $\varphi_c$ if, and only if, $d$ is of a particular form.
\begin{equation}\label{dphi}
d\varphi_c=-\varphi_c d \iff d(\varphi_c(a))=-\varphi_c(d(a)) \quad \textrm{for all } a \in A
\end{equation}
The LHS is
\begin{eqnarray*}
d(\varphi_c(a))&=&d(c^{-1}ac)\\
&=& d(c^{-1})ac+\varphi_c(c^{-1})d(ac)\\
&=& d(c^{-1})ac+c^{-1}c^{-1}c[d(a)c+\varphi_c(a)d(c)]\\
&=& d(c^{-1})ac+c^{-1}d(a)c+c^{-1}\varphi_c(a)d(c).\\
\end{eqnarray*}
The RHS of \eqref{dphi} gives $-\varphi_c(d(a))=-c^{-1}d(a)c$.
Thus, we can easily conclude that $d$ anticommutes with $\varphi_c$ if, and only if,
\[-c^{-1}d(a)c=d(c^{-1})ac+c^{-1}d(a)c+c^{-1}\varphi_c(a)d(c) \quad \textrm{ for all } a \in A,\]
that is
\[d(a)=\frac{-cd(c^{-1})a-\varphi_c(a)d(c)c^{-1}}{2} \quad \textrm{ for all } a \in A.\]
Since $d$ is a $\varphi_c$-derivation, it is immediate to see that $d(c)c^{-1}+cd(c^{-1})=d(c)c^{-1}+\varphi_c(c)d(c^{-1})=d(cc^{-1})=d(1)=0$, and therefore that $cd(c^{-1})=-d(c)c^{-1}$, hence
\[
d(a)=\frac{d(c)c^{-1}}{2}a-\varphi_c(a)\frac{d(c)c^{-1}}{2} \quad \textrm{for all } a \in A.
\]
By setting  $u_i:=\frac{d_i(c)c^{-1}}{2}$, we have proved that each $d_i$ satisfies \eqref{dphi} if, and only if, 
\begin{equation}\label{derisinn}
d_i(a)=u_ia-\varphi_c(a)u_i \quad \textrm{for all } a \in A.
\end{equation}
In particular every $d_i$ must be an inner $\varphi_c$-derivation.
Notice that \eqref{derisinn} for $a=c$ gives
\[2u_ic\overset{def. \ u_i}{=}d_i(c)\overset{\eqref{derisinn}}{=}u_ic-\varphi_c(c)u_i=u_ic-cu_i \iff 2u_ic=u_ic-cu_i \iff u_ic+cu_i=0,\]
so that we have determined an $n+1$-uple of elements $(c, u_1, \ldots, u_n)$ satisfying
\begin{eqnarray}
&c^2 \in k,\label{c^2}\\
&u_ic+cu_i=0 \ \textrm{for every } i=1, \ldots, n \label{uc}.
\end{eqnarray}
Let us prove that also
\begin{eqnarray}
&u_i^2 \in k \ \textrm{for every } i=1, \ldots, n,\label{u^2}\\
&u_iu_j+u_ju_i \in k  \ \textrm{for every } i,j=1, \ldots, n,\label{uiuj}
\end{eqnarray}
are satisfied.
First we point out that
\begin{eqnarray}
\varphi_c(u_i)&=&\varphi_c\left(\frac{d_i(c)c^{-1}}{2} \right)=\frac{1}{2}\varphi_c(d_i(c))\varphi_c(c^{-1})=-\frac{1}{2}d_i(\varphi_c(c))c^{-1}=-\frac{1}{2}d_i(c)c^{-1}=-u_i \label{phiu}
\end{eqnarray}
for every $i=1, \ldots, n$.
Then we observe that, for every $a \in A$
\begin{eqnarray*}
d_i^2(a)=0 &\iff & d_i(u_ia-\varphi_c(a)u_i)=0\\
& \iff & u_i^2a-\varphi_c(u_ia)u_i-u_i\varphi_c(a)u_i+\varphi_c(\varphi_c(a)u_i)u_i=0\\
& \iff &  u_i^2a-\varphi_c(u_i)\varphi_c(a)u_i-u_i\varphi_c(a)u_i+\varphi_c^2(a)\varphi_c(u_i)u_i=0\\
&\overset{\eqref{phiu}+\eqref{varphi^2}}{\iff} & u_i^2a+u_i\varphi_c(a)u_i-u_i\varphi_c(a)u_i-au_i^2=0\\
&\iff & u_i^2a-au_i^2=0.
\end{eqnarray*}
This is equivalent to $u_i^2 \in \mathcal{Z}(A)=k$, thus \eqref{u^2} holds. Finally, for any two $\varphi_c$-derivations $d_i$, $d_j$, we have $d_id_j=-d_jd_i$ if, and only if, for every $a \in A$
\begin{eqnarray*}
&&d_i(d_j(a))=-d_j(d_i(a))\\
&\iff& u_id_j(a)-\varphi_c(d_j(a))u_i=-u_jd_i(a)+\varphi_c(d_i(a))u_j\\
&\iff & u_iu_j a-u_i\varphi_c(a)u_j-\varphi_c(u_ja-\varphi_c(a)u_j)u_i=\\
&&=-u_ju_ia+u_j\varphi_c(a)u_i+\varphi_c(u_ia-\varphi_c(a)u_i)u_j\\
&\overset{\eqref{phiu}+\eqref{varphi^2}}{\iff}& u_iu_j a-u_i\varphi_c(a)u_j+u_j\varphi_c(a)u_i-au_ju_i=-u_ju_ia+u_j\varphi_c(a)u_i-u_i\varphi_c(a)u_j+au_iu_j\\
&\iff & u_iu_j a-au_ju_i=-u_ju_ia+au_iu_j\\
&\iff & (u_iu_j+u_ju_i)a=a(u_iu_j+u_ju_i).
\end{eqnarray*}
Then $u_iu_j+u_ju_i \in \mathcal{Z}(A)=k$ and \eqref{uiuj} is verified. In this way we have shown that, when $A$ is central simple, to a $n+1$-uple of maps $(\varphi, d_1, \ldots, d_n)$ satisfying \eqref{varphi^2}-\eqref{didj} corresponds a $n+1$-uple of elements $(c,u_1, \ldots, u_n)$ for which \eqref{c^2}-\eqref{uiuj} are verifed.

\medskip

Conversely, consider a tuple $(c,u_1, \ldots, u_n)$ such that \eqref{c^2}-\eqref{uiuj} hold true. If $\varphi_c:A \rightarrow A$ is defined by $\varphi_c(a)=c^{-1}ac$ for every $a \in A$, with $c$ satisfying \eqref{c^2}, then it is immediate to check that $\varphi_c$ is an automorphism such that $\varphi_c^2 \equiv \textrm{Id}$. Moreover, given the map $d_i: A \rightarrow A$ defined by $d_i(a)=u_ia-\varphi_c(a)u_i$ for every $a \in A$, we can easily show that it satisfies \eqref{varphid}.
In fact we have
\begin{eqnarray*}
\varphi_c(d_i(a))&=&\varphi_c(u_ia-\varphi_c(a)u_i)\\
&\overset{\eqref{c^2}}{=}&\varphi_c(u_i)\varphi_c(a)-a\varphi_c(u_i)\\
&=&c^{-1}u_iac-ac^{-1}u_ic\\
&\overset{\eqref{uc}}{=}& -u_ic^{-1}ac+au_i\\
&=&-d_i(c^{-1}ac)\\
&=&-d_i\varphi_c(a)
\end{eqnarray*}
for every $a \in A$.
Furthermore
\begin{eqnarray*}
d_i(d_j(a)) &=& u_id_j(a)-\varphi_c(d_j(a))u_i\\
&=& u_i(u_ja-\varphi_c(a)u_j)-c^{-1}(u_ja-\varphi_c(a)u_j)cu_i\\
&=& u_iu_ja-u_i\varphi_c(a)u_j-c^{-1}u_ja c u_i-c^{-1}\varphi_c(a)u_jcu_i\\
&=& u_iu_ja-u_ic^{-1}acu_j-c^{-1}u_ja c u_i-c^{-2}a c u_jcu_i\\
&\overset{\eqref{c^2}+\eqref{uc}}{=}& u_iu_ja+c^{-1}u_iacu_j+u_jc^{-1}a c u_i+a  u_ju_i\\
&\overset{\eqref{uiuj}}{=}&-u_ju_ia+c^{-1}u_iacu_j+u_jc^{-1}a c u_i+au_iu_j\\
&\overset{\eqref{c^2}+\eqref{uc}}{=}& -u_ju_ia+u_jc^{-1}acu_i+c^{-1}u_iacu_j-c^{-2}acu_icu_j\\
&=& -u_ju_ia+u_j\varphi_c(a)u_i+c^{-1}u_iacu_j-c^{-1}\varphi_c(a)u_icu_j\\
&=& -u_j(u_ia-\varphi_c(a)u_i)+c^{-1}(u_ia-\varphi_c(a)u_i)cu_j\\
&=& -d_j(d_i(a))
\end{eqnarray*}
for every $a \in A$, so \eqref{didj} holds true.
Finally,
\begin{eqnarray*}
d_i^2(a)&=&d_i(u_ia-\varphi_c(a)u_i)\\
&=& u_i^2a-u_i\varphi_c(a)u_i-\varphi_c(u_ia)u_i+\varphi_c^2(a)\varphi_c(u_i)u_i\\
&\overset{\eqref{c^2}}{=}& u_i^2a-u_ic^{-1}acu_i-c^{-1}u_iacu_i+ac^{-1}u_icu_i\\
&\overset{\eqref{uc}}{=}& u_i^2a-au_i^2\\
&\overset{\eqref{u^2}}{=}& 0
\end{eqnarray*}
for every $a \in A$, thus \eqref{d^2} is verified.
\end{proof}

\begin{remark}\label{bijcorr}
The correspondence
\[
\lbrace (c,u_1, \ldots, u_n) \ | \ \textrm{satisfying } \eqref{c^2}-\eqref{uiuj} \rbrace \Longleftrightarrow
\lbrace (\varphi_c, d_1, \ldots, d_n) \ | \ \textrm{satisfying } \eqref{varphi^2}-\eqref{didj} \rbrace
\]
is surjective but not injective. This is a direct consequence of the fact that $\textrm{Inn}(A) \cong \frac{\textrm{U}(A)}{\mathcal{Z}(\textrm{U}(A))}$.
In fact we have
\[\varphi_{c_1}=\varphi_{c_2} \iff c_1\mathcal{Z}(A)=c_2\mathcal{Z}(A).\] Let us indicate this coset with $c\mathcal{Z}(A)$ and the associated inner morphism with $\varphi_c$.
Now observe that any two $\varphi_c$-derivation $d$ and $d'$ are equal if, and only if, 
\begin{eqnarray*}
d(a)=d'(a) &\iff & ua-\varphi_c(a)u=va-\varphi_c(a)v\\
&\iff & ua-c^{-1}acu=va-c^{-1}acv\\
&\iff & c(u-v)a=ac(u-v)
\end{eqnarray*}
for every $a \in A$. Therefore $d=d' \iff c(u-v) \in \mathcal{Z}(A) \iff (u-v) \in c^{-1}\mathcal{Z}(A)$. When \eqref{c^2} holds, we have $c^{-1}\mathcal{Z}(A)=c\mathcal{Z}(A)$.
Then we can define a relation $\sim$ over the set $\lbrace (c,u_1, \ldots, u_n) \ | \ \textrm{satisfying } \eqref{c^2}-\eqref{uiuj} \rbrace$ in the following way:
\begin{equation}\label{rel1}
(c',u_1, \ldots, u_n) \sim (c'',v_1, \ldots, v_n) \iff  c'\mathcal{Z}(A)=c''\mathcal{Z}(A) \ni (u_i-v_i) \textrm{ for every } i=1, \ldots n.
\end{equation}
It is easy to prove that $\sim$ is an equivalence relation and that in this way we retrieve a bijective correspondence
\[
\lbrace (c,u_1, \ldots, u_n) \ | \ \textrm{satisfying } \eqref{c^2}-\eqref{uiuj} \rbrace /\sim\\
 \quad \cong \quad 
\lbrace (\varphi, d_1, \ldots, d_n) \ | \ \textrm{satisfying } \eqref{varphi^2}-\eqref{didj} \rbrace.\]
\end{remark}

\subsection{The case \texorpdfstring{$n$}{n} even}

If $n$ is even, then a semisimple Clifford algebra $A=Cl(\alpha, \beta_i, \gamma_i, \lambda_{ij})$ can be written as $A=A_0 \oplus A_0z$, where $A_0$ denotes the even part of $A$ and $z$ is the pseudoscalar of $A$. Remember that $z$ is defined as
\[z=e_1e_2 \ldots e_{n+1},\]
where $e_1$, $e_2$, \ldots, $e_{n+1}$ denote an orthogonal basis for the quadratic form $q$ associated to $A$. In this case, the pseudoscalar is a non-zero element of $\mathcal{Z}(A) \cap A_1$. Moreover we have
\[z^2=(-1)^{\frac{n(n+1)}{2}}\det Q \neq 0\]
and therefore $z$ is invertible.
For further reference on this, see \cite[Ch. V, Sec. 2]{La}. A Clifford algebra also admits a so-called \emph{main involution} (or \emph{grade involution}) \cite[p. 93]{La} 
\begin{eqnarray*}
\sigma: A_0 \oplus A_1 &\rightarrow& A_0 \oplus A_1\\
t+s &\mapsto& t-s
\end{eqnarray*}
(when $k=\mathbb{R}$, $n=0$ and $\det Q=-1$ this coincides with the usual complex conjugation). 
One can easily prove that the map $\sigma$ defined above is an algebra automorphism of order $2$.
%
The statement of \autoref{class1} encourage us to study the involutions of $A$ by making a distinction between two cases, depending on whether the ground field $k$ does or does not contain a square root of the non-zero scalar $\delta:=z^2=(-1)^{\frac{n(n+1)}{2}}\det Q$.

\subsubsection{If $n$ is even and $\delta \notin k^2$}

If $\delta \notin k^2$, then $\mathcal{Z}(A)=k \oplus k z \cong k(\sqrt{\delta} )$, and $A=A_0 \oplus A_0z$ is a central simple algebra over $k(\sqrt{\delta})$ \cite[Th. V.2.4]{La}. 

Let us consider an automorphism $f \in \textrm{Aut}_k(A)$. We have $f\left(\mathcal{Z}(A) \right) \subseteq \mathcal{Z}(A)$ (since $f$ is surjective) and $f(\eta)=\eta$ for every $\eta \in k$, since $f$ is $k$-linear and $f(1)=1$ by definition. Let $\mathcal{Z}(A) \ni w:=f(z)=\mu_1+\mu_2z$ with $\mu_1,\mu_2 \in k$. Then
\[\mu_1^2+2\mu_1\mu_2z+\mu_2^2\delta=(\mu_1+\mu_2z)^2=\left( f(z)\right)^2=f(z^2)=f(\delta)=\delta.\]
By linear independence, we immediately find $\mu_1^2+(\mu_2^2-1)\delta=0=\mu_1\mu_2$. If $\mu_1 \neq 0$, then $\mu_2=0$, but this means $f(z) \in k$, hence $f\left(\mathcal{Z}(A) \right)=k$, which is a contradiction since $f$ is invertible. Therefore we must have $\mu_1=0$ and $\mu_2^2=1$. Given that $k$ is a field with $\chara k \neq 2$, this is equivalent to $\mu_2=\pm 1$.

\medskip

If $\mu_2=1$, then the restriction of $f$ to $\mathcal{Z}(A)=k \oplus kz$ is the identity and we find that
\[f(ra)=f(r)f(a)=rf(a)\]
for every $r \in \mathcal{Z}(A) \cong k(\sqrt{\delta})$ and every $a \in A$. This means that $f$ is a $k(\sqrt{\delta})$-linear automorphism of $A$ and, by the Skolem-Noether Theorem, it must be inner.

\medskip

If $\mu_2=-1$, then the restriction of $\sigma f$ to $\mathcal{Z}(A)$ is the identity and this time we see that $\sigma f$ is a $k(\sqrt{\delta})$-linear automorphism and that therefore it must be inner. This means that $f$ is the composition of $\sigma$ and an inner automorphism of $A$. Hence, we deduce that 
\[\textrm{Aut}_k(A)=\lbrace \varphi_c, \sigma \varphi_c \ | \ c \in \textrm{U}(A) \rbrace,\]
where $\varphi_c$ is again used to denote the inner automorphism associated to the element $c \in A$.

\subsubsection{If $n$ is even and $\delta \in k^2$}

If $\delta \in k^2$, then $\mathcal{Z}(A)=k \oplus kz \cong k \times k$, $A=A_0 \oplus A_0z \cong A_0 \times A_0$ and $A_0$ is a central simple algebra over $k$ \cite[Th. V.2.4]{La}. $A$ has only two non-trivial central idempotents. In fact, let $\mu_1+\mu_2z \in \mathcal{Z}(A)$ for some $\mu_1, \mu_2 \in k$. Then
\[\mu_1^2+\mu_2^2\delta+2\mu_1\mu_2z=(\mu_1+\mu_2z)^2=\mu_1+\mu_2z\]
if, and only if, $\mu_1^2+\mu_2^2\delta=\mu_1$ and $2 \mu_1\mu_2=\mu_2$. One easily sees that the solutions to this couple of equations are 
\[0=0+0z, \ 1=1+0z, t_1:=\frac{1}{2}+\frac{1}{2\sqrt{\delta}}z, \ t_2:=\frac{1}{2}-\frac{1}{2\sqrt{\delta}}z.\]
Then $A_0 \oplus A_0z=A= A_0t_1 \oplus A_0t_2 \cong A_0^{(1)} \times A_0^{(2)}$, via the isomorphism $\Phi:A \rightarrow A_0^{(1)} \times A_0^{(2)}$ 
\[a_0+b_0z=(a_0+\sqrt{\delta}b_0)t_1+(a_0-\sqrt{\delta}b_0)t_2 \longmapsto (a_0+\sqrt{\delta}b_0,a_0-\sqrt{\delta}b_0).\]
Here we have introduced the superscripts on the two copies of $A_0$ to remark the fact that they are distinct, though isomorphic.
Notice that $t_1 \mapsto (1,0)$ and $t_2 \mapsto (0,1)$.
The inverse of $\Phi$ is given by $\Psi: A_0^{(1)} \times A_0^{(2)} \rightarrow A$
\[(u,v) \mapsto \frac{u+v}{2}+\frac{u-v}{2\sqrt{\delta}}z.\]
%
\begin{remark}\label{tausigma}
Observe that $\Psi\tau=\sigma \Psi$, where $\tau=A_0^{(2)} \times A_0^{(1)} \rightarrow A_0^{(1)} \times A_0^{(2)}$ is the twist defined by $\tau(u,v)=(v,u)$ and $\sigma$ is the main involution of $A$.
\end{remark}
Fix an automorphism $f \in \textrm{Aut}_k(A_0^{(1)} \times A_0^{(2)})$. Clearly $f$ sends $(0,0)$ to $(0,0)$ and $(1,1)$ to $(1,1)$, therefore we are left with two possibilities.
\begin{enumerate}
\item If $f(1,0)=(1,0)$, then $f(0,1)=(0,1)$  and we have $f=f_1 \times f_2$, where $f_i:=f|_{A_0^{(i)}}^{A_0^{(i)}}:A_0^{(i)} \rightarrow A_0^{(i)}$. Given that $A_0$ is a central simple algebra over $k$, then the $f_i$'s must be inner automorphisms by the Skolem-Noether Theorem.
This means that $f:A_0^{(1)} \times A_0^{(2)} \rightarrow A_0^{(1)} \times A_0^{(2)}$ is given by
\[f_{c_1c_2}(u,v):=f(u,v)=(f_1(u),f_2(v))=(c_1^{-1}uc_1,c_2^{-1}vc_2)=(c_1,c_2)^{-1}(u,v)(c_1,c_2)\]
for some invertible $c=(c_1, c_2) \in A_0^{(1)} \times A_0^{(2)}$.
The corresponding automorphism of $A$ is $\varphi_c:=\Psi f_{c_1c_2} \Phi$:
\begin{eqnarray*}
\varphi_c(a)&=&\Psi f_{c_1c_2} \Phi(a)\\
&=& \Psi(c)^{-1}a\Psi(c)\\
&=& \left(\frac{c_1+c_2}{2}+\frac{c_1-c_2}{2\sqrt{\delta}}z \right)^{-1} \cdot a \cdot  \left(\frac{c_1+c_2}{2}+\frac{c_1-c_2}{2\sqrt{\delta}}z \right).
\end{eqnarray*}
Since the invertibility of $c_1$ and $c_2$ is equivalent to that of $\frac{c_1+c_2}{2}+\frac{c_1-c_2}{2\sqrt{\delta}}z$ we will write 
\[\varphi_c(a)=c^{-1}ac \quad \textrm{for every } a \in A\]
with abuse of notation, regarding $c$ as an invertible element of $A$.
\item If $f(1,0)=(0,1)$, then $f(0,1)=(1,0)$ and we have $\tau f=f_1 \times f_2$, where again $\tau$ indicates the appropriate twist. We have that the $f_i$'s must be inner automorphisms of $A_0$, by the Skolem-Noether Theorem, and therefore
\[f_{c_1c_2}(u,v):=f(u,v)=\tau(f_1(u),f_2(v))=(c_2^{-1}vc_2, c_1^{-1}uc_1)=\tau((c_1,c_2)^{-1}(u,v)(c_1,c_2))\]
for some invertible $c=(c_1, c_2) \in A_0^{(1)} \times A_0^{(2)}$.
The corresponding automorphism of $A$ is
\begin{eqnarray*}
\varphi'_c(a)&=&\Psi f_{c_1c_2} \Phi(a)\\
&=&\Psi (\tau(c^{-1}\Phi(a)c))\\
&=&\Psi \tau(c^{-1})\Psi\tau(\Phi(a)) \Psi\tau(c)\\
&\overset{\ref{tausigma}}{=}&\sigma\Psi (c^{-1})\sigma(a) \sigma\Psi(c)\\
&=&\sigma \left(\left(\frac{c_1+c_2}{2}+\frac{c_1-c_2}{2\sqrt{\delta}}z \right)^{-1} \cdot a \cdot  \left(\frac{c_1+c_2}{2}+\frac{c_1-c_2}{2\sqrt{\delta}}z \right)\right).
\end{eqnarray*}
Again, the invertibility of $c_1$ and $c_2$ is equivalent to that of $\frac{c_1+c_2}{2}+\frac{c_1-c_2}{2\sqrt{\delta}}z$, so we will write 
\[\varphi'_c(a)=\sigma(c^{-1}ac)=\sigma\varphi_c(a) \quad \textrm{for every } a \in A\]
with abuse of notation, regarding $c$ as an invertible element of $A$.
\end{enumerate}

Therefore we can conclude that 
\[\textrm{Aut}_k(A)=\lbrace \varphi_c, \sigma\varphi_c \ | \ c \in \textrm{U}(A) \rbrace,\]
as was the case when $\delta \notin k^2$.
\begin{remark}
We have proved that for a $2^{n+1}$-dimensional semisimple Clifford algebra $A=Cl(\alpha, \beta_i, \gamma_i, \lambda_{ij})$ over a field $k$ with $\chara k \neq 2$ 
\[\textrm{Aut}_k(A)=\lbrace \varphi_c \ | \ c \in \textrm{U}(A) \rbrace\] if $n$ is odd
and
\[\textrm{Aut}_k(A)=\lbrace \varphi_c, \sigma\varphi_c \ | \ c \in \textrm{U}(A) \rbrace\]
if $n$ is even.
\end{remark}

\begin{theorem}\label{evencase}
Let $A=Cl(\alpha, \beta_i, \gamma_i, \lambda_{ij})$ be a $2^{n+1}$-dimensional Clifford algebra with $n$ even. Assume that $A$ is semisimple. Let $z$ denote its pseudoscalar and $\sigma$ the main involution of $A$. Then an $E(n)$-coaction on $A$ is either given by:
\[\rho(a)= \underset{P \subseteq \lbrace 1, \ldots, n \rbrace}{\sum}(-1)^{\lfloor \frac{|P|+1}{2} \rfloor} \left[ d_P(a) \otimes \frac{x_P +(-1)^{|P|}gx_P}{2}+c^{-1}d_P(a)c \otimes \frac{x_P +(-1)^{|P|+1}gx_P}{2} \right]\]
where
\begin{enumerate}
\item $c$ is an invertible element of $A$ such that $c^2 \in k \oplus kz$,
\item $d_{\emptyset}=\textrm{Id}_A$ and $d_P=d_{i_1}d_{i_2}\cdots d_{i_{|P|}}$ is a composition of maps $d_i:A \rightarrow A$ defined by $d_i(a)=u_ia-c^{-1}acu_i$ such that $u_ic+cu_i=0$, $u_i^2 \in k \oplus kz$ and $u_iu_j+u_ju_i \in k \oplus kz$,
\end{enumerate}
or by:
\[\rho(a)= \underset{P \subseteq \lbrace 1, \ldots, n \rbrace}{\sum}(-1)^{\lfloor \frac{|P|+1}{2} \rfloor} \left[ d_P(a) \otimes \frac{x_P +(-1)^{|P|}gx_P}{2}+\sigma(c^{-1}d_P(a)c) \otimes \frac{x_P +(-1)^{|P|+1}gx_P}{2} \right]\]
where
\begin{enumerate}
\item $c$ is an invertible element of $A$ such that $c\sigma(c) \in k \oplus kz$,
\item $d_{\emptyset}=\textrm{Id}_A$ and $d_P=d_{i_1}d_{i_2}\cdots d_{i_{|P|}}$ is a composition of maps $d_i:A \rightarrow A$ defined by $d_i(a)=u_ia-\sigma(c^{-1}ac)u_i$ such that $u_ic+c\sigma(u_i)=0$, $u_i^2 \in k \oplus k z$ and $u_iu_j+u_ju_i \in k \oplus k z$.
\end{enumerate}
\end{theorem}

\begin{proof}
By performing the same steps contained in the proof of \autoref{oddcase} we can prove that $n+1$-uples $(\varphi_c, d_1, \ldots, d_n)$ satisfying \eqref{varphi^2}-\eqref{didj} are in correspondence with $n+1$-uples of elements $(c,u_1, \ldots, u_n)$ satisfying 
\begin{eqnarray}
&c^2 \in k\oplus kz,\label{c^2bis}\\
&u_ic+cu_i=0 \ \textrm{for every } i=1, \ldots, n \label{ucbis},\\
&u_i^2 \in k\oplus kz \ \textrm{for every } i=1, \ldots, n,\label{u^2bis}\\
&u_iu_j+u_ju_i \in k\oplus kz  \ \textrm{for every } i,j=1, \ldots, n.\label{uiujbis}
\end{eqnarray}
We just need to pay attention to the fact that in this instance $\mathcal{Z}(A)=k\oplus kz$.

\medskip

Let us now consider a $n+1$-uple $(\sigma\varphi_c, d_1, \ldots, d_n)$ satisfying \eqref{varphi^2}-\eqref{didj}. We are going to prove that it corresponds to a tuple of elements $(c, u_1, \ldots, u_n)$ such that
\begin{eqnarray}
&c\sigma(c) \in k\oplus kz\label{sigmac}\\
&u_ic+c\sigma(u_i)=0 \ \textrm{for every } i=1, \ldots, n \label{sigmauc}\\
&u_i^2 \in k \oplus kz \ \textrm{for every } i=1, \ldots, n\label{sigmau^2}\\
&u_iu_j+u_ju_i \in k \oplus kz  \ \textrm{for every } i,j=1, \ldots, n\label{sigmauiuj}.
\end{eqnarray}
First of all
\begin{eqnarray*}
(\sigma\varphi_c)^2 = \textrm{Id}_A &\iff & \sigma(c^{-1})c^{-1}ac\sigma(c)=a \quad \textrm{for all } a \in A \\
&\iff & ac\sigma(c)=c \sigma(c)a \quad \textrm{for all } a \in A \\
&\iff & c \sigma(c) \in \mathcal{Z}(A)=k \oplus kz,
\end{eqnarray*}
so \eqref{sigmac} is satisfied.
Next, suppose $d:A \rightarrow A$ is a $\sigma\varphi_c$-derivation. Since $z \in \mathcal{Z}(A)$, we have
\[d(a)z+\sigma\varphi_c(a)d(z)=d(az)=d(za)=d(z)a+\sigma \varphi_c(z)d(a)=d(z)b+\sigma (c^{-1}zc)d(a)=d(z)a-zd(a),\]
for every $a \in A$. This means that
\[2zd(a)=d(z)a-\sigma\varphi_c(a)d(z) \quad \textrm{for all } a \in A,\]
or equivalently 
\[d(a)=\frac{d(z)z^{-1}}{2} a-\sigma\varphi_c(a)\frac{d(z)z^{-1}}{2} \quad \textrm{for all } a \in A.\]
Therefore $d$ is a $\sigma\varphi_c$-derivation only if 
\begin{equation}\label{sigmaphider}
d(a)=ua-\sigma\varphi_c(a)u
\end{equation}
for every $a \in A$, where $u:=\frac{d(z)z^{-1}}{2}$. In particular $d$ must be an inner $\sigma\varphi_c$-derivation. 
Therefore we have $d_i(a)=u_ia-\sigma\varphi_c(a)u_i$ for every $a \in A$ and any $i=1, \ldots, n$.
Furthermore, since \eqref{varphid} holds, we have
\begin{eqnarray*}
-2u_iz&=&-u_iz+\sigma\varphi_c(z)u_i\\
&=&d(-z)\\
&=&d(\sigma(c^{-1}zc))\\
&=&d\sigma\varphi_c(z)\\
&\overset{\eqref{varphid}}{=}&-\sigma\varphi_cd(z)\\
&\overset{def. \ u_i}{=}&-\sigma\varphi_c(2u_iz)\\
&=&-2\sigma(c^{-1}u_izc),
\end{eqnarray*}
that is $u_iz=\sigma(c^{-1}u_izc)$, for any $i=1, \ldots, n$. 
If we apply $\sigma$ and we multiply by $c$ on the left on both sides of the equation we get
\[-c\sigma(u_i)z=u_izc \quad \textrm{for any }i=1, \ldots, n.\]
Since $z \in \mathcal{Z}(A)$ and it is also invertible, we can conclude that
\[-c\sigma(u_i)z=u_izc \iff -c\sigma(u_i)z=u_icz \iff (u_ic+c\sigma(u_i))z=0 \iff u_ic+c\sigma(u_i)=0,\]
for any $i=1, \ldots, n$. This shows that \eqref{sigmauc} holds true.

Next, observe that
\begin{equation}\label{sigmaphiu}
\sigma\varphi_c(u_i)=\sigma\varphi_c\left(\frac{d(z)z^{-1}}{2} \right)=\frac{1}{2}\sigma\varphi_c(d(z))\sigma\varphi_c(z^{-1})=\frac{1}{2}d(\sigma\varphi_c(z))z^{-1}=-\frac{1}{2}d(z)z^{-1}=-u_i
\end{equation}
for every $i=1, \ldots, n$.
Then 
\begin{eqnarray*}
d_i^2(a)=0 &\iff & d_i(u_ia-\sigma\varphi_c(a)u_i)=0\\ 
&\iff & u_i^2a-\sigma\varphi_c(u_ia)u_i-u_i\sigma\varphi_c(a)u_i+\sigma\varphi_c(\sigma\varphi_c(a)u_i)u_i=0\\
&\iff & u_i^2a-\sigma\varphi_c(u_i)\sigma\varphi_c(a)u_i-u_i\sigma\varphi_c(a)u_i+(\sigma\varphi_c)^2(a)\sigma\varphi_c(u_i)u_i=0\\
&\overset{\eqref{sigmaphiu}}{\iff}& u_i^2a+u_i\sigma\varphi_c(a)u_i-u_i\sigma\varphi_c(a)u_i-au_i^2=0\\
&\iff & u_i^2a-au_i^2=0
\end{eqnarray*}
for every $a \in A$. Therefore \eqref{d^2} implies \eqref{sigmau^2}.
Finally, for any two $d_i$, $d_j$ and any $a \in A$ we have 
\begin{eqnarray*}
d_i(d_j(a))=-d_j(d_i(a)) &\iff & u_id_j(a)-\sigma\varphi_c(d_j(a))u_i=-u_jd_i(a)+\sigma\varphi_c(d_i(a))u_j\\
&\iff & u_iu_j a-u_i\sigma\varphi_c(a)u_j-\sigma\varphi_c(u_ja-\sigma\varphi_c(a)u_j)u_i=\\
&=&-u_ju_ia+u_j\sigma\varphi_c(a)u_i+\sigma\varphi_c(u_ia-\sigma\varphi_c(a)u_i)u_j \\
&\overset{\eqref{varphi^2}+\eqref{sigmaphiu}}{\iff}& u_iu_j a-u_i\sigma\varphi_c(a)u_j+u_j\sigma\varphi_c(a)u_i-au_ju_i=\\
&=&-u_ju_ia+u_j\sigma\varphi_c(a)u_i-u_i\sigma\varphi_c(a)u_j+au_iu_j\\
&\iff & u_iu_j a-au_ju_i=-u_ju_ia+au_iu_j\\
 &\iff & (u_iu_j+u_ju_i)a=a(u_iu_j+u_ju_i).
\end{eqnarray*}
Thus \eqref{didj} implies that \eqref{sigmauiuj} is verified.

Conversely, consider a $n+1$-uple of elements $(c, u_1, \ldots, u_n)$ satisfying \eqref{sigmac}-\eqref{sigmauiuj}. Then the tuple $(\sigma \varphi_c, d_1, \ldots, d_n)$, where $\varphi_c(a)=c^{-1}ac$ for every $a \in A$ and $d_i(a)=u_ia-\sigma\varphi_c(a)u_i$ for every $a \in A$ and $i=1, \ldots, n$, satisfies \eqref{varphi^2}-\eqref{didj}.
In fact
\[(\sigma \varphi_c)^2(a)=\sigma(c^{-1})c^{-1}ac\sigma(c) \overset{\eqref{sigmac}}{=}\sigma(c^{-1})c^{-1}c\sigma(c)a=a\]
for every $a \in A$ and thus $(\sigma \varphi_c)^2=\textrm{Id}_A$.
Moreover 
\begin{eqnarray*}
\sigma\varphi_c(d_i(a))&=&\sigma\varphi_c(u_ia-\sigma\varphi_c(a)u_i)\\
&\overset{\eqref{sigmac}}{=}&\sigma\varphi_c(u_ia)-a(\sigma\varphi_c(u_i))\\
&=&\sigma(c^{-1}u_iac)-a(\sigma(c^{-1}u_ic))\\
&\overset{\eqref{sigmauc}}{=}& -u_i\sigma(c^{-1}ac)+au_i\\
&\overset{\eqref{sigmac}}{=}&-d_i(\sigma(c^{-1}ac))\\
&=&-d_i\sigma\varphi_c(a)
\end{eqnarray*}
for every $a \in A$ and any $i=1, \ldots, n$, so \eqref{varphid} holds true. We also have
\begin{eqnarray*}
d_i(d_j(a)) &=& u_id_j(a)-\sigma\varphi_c(d_j(a))u_i\\
&=& u_i(u_ja-\sigma\varphi_c(b)u_j)-\sigma\varphi_c(u_ja-\sigma\varphi_c(a)u_j)u_i\\
&\overset{\eqref{sigmac}}{=}& u_iu_ja-u_i\sigma\varphi_c(a)u_j-\sigma\varphi_c(u_ja)u_i+a\sigma\varphi_c(u_j)u_i\\
&\overset{\eqref{sigmauc}}{=}& u_iu_ja-u_i\sigma\varphi_c(a)u_j+u_j\sigma\varphi_c(a)u_i-bu_ju_i\\
&\overset{\eqref{sigmauiuj}}{=}& -u_ju_ia+u_j\sigma\varphi_c(a)u_i-u_i\sigma\varphi_c(a)u_j+bu_iu_j\\
&\overset{\eqref{sigmauc}}{=}& -u_ju_ia+u_j\sigma\varphi_c(a)u_i+\sigma\varphi_c(u_ia)u_j-a\sigma\varphi_c(u_i)u_j\\
&\overset{\eqref{sigmac}}{=}& -u_j(u_ia-\sigma\varphi_c(a)u_i)+\sigma\varphi_c(u_ia-\sigma\varphi_c(a)u_i)u_j\\
&=& -d_j(d_i(a))
\end{eqnarray*}
for any $a \in A$, so that \eqref{didj} is verified. Finally
\begin{eqnarray*}
d_i^2(a)&=&d_i(u_ia-\sigma\varphi_c(a)u_i)\\
&=& u_i^2a-u_i\sigma\varphi_c(a)u_i-\sigma\varphi_c(u_ia)u_i+(\sigma\varphi_c)^2(a)\sigma\varphi_c(u_i)u_i\\
&\overset{\eqref{sigmac}}{=}& u_i^2a-u_i\sigma\varphi_c(a)u_i-\sigma\varphi_c(u_ia)u_i+a\sigma\varphi_c(u_i)u_i\\
&\overset{\eqref{sigmauc}}{=}& u_i^2a-u_i\sigma\varphi_c(a)u_i+u_i\sigma\varphi_c(a)u_i-au_i^2\\
&=& u_i^2a-au_i^2\\
&\overset{\eqref{sigmau^2}}{=}&0
\end{eqnarray*}
for any $a \in A$ and \eqref{d^2} is proved.
\end{proof}

\begin{remark}
As was the case with $n$ odd, the correspondence
\[
\lbrace (c,u_1, \ldots, u_n) \ | \ \textrm{satisfying } \eqref{c^2bis}-\eqref{uiujbis} \rbrace \Longleftrightarrow
\lbrace (\varphi_c, d_1, \ldots, d_n) \ | \ \textrm{satisfying } \eqref{varphi^2}-\eqref{didj} \rbrace
\]
is not injective, but it can be made so by introducing the relation $\sim$ defined in \eqref{rel1}:
\[
\lbrace (c,u_1, \ldots, u_n) \ | \ \textrm{satisfying } \eqref{c^2bis}-\eqref{uiujbis} \rbrace/\sim \\ \quad \cong
\quad \lbrace (\varphi_c, d_1, \ldots, d_n) \ | \ \textrm{satisfying } \eqref{varphi^2}-\eqref{didj} \rbrace.
\]
Similarly the correspondence
\[
\lbrace (c,u_1, \ldots, u_n) \ | \ \textrm{satisfying } \eqref{sigmac}-\eqref{sigmauiuj} \rbrace \Longleftrightarrow
\lbrace (\sigma\varphi_c, d_1, \ldots, d_n) \ | \ \textrm{satisfying } \eqref{varphi^2}-\eqref{didj} \rbrace
\] is not injective. 
In fact it is easy to see that we have
\[\sigma\varphi_{c_1}=\sigma\varphi_{c_2} \iff c_1\mathcal{Z}(A)=c_2\mathcal{Z}(A).\]
Let us indicate this coset with $c\mathcal{Z}(A)$ and the associated morphism with $\sigma\varphi_c$.
If two $\sigma\varphi_c$-derivation $d_i$ and $d_i'$ are equal, then we must have $d_i(z)=d_i'(z)$. But remember that $d_i(z)=2u_iz$ (see the proof of \autoref{evencase}) and therefore $d_i=d'_i \implies u_i=u'_i$, because $z$ is invertible. The converse implication is trivial.
Then we can define a relation $\sim'$ over the set $\lbrace (c,u_1, \ldots, u_n) \ | \ \textrm{satisfying } \eqref{sigmac}-\eqref{sigmauiuj} \rbrace$ in the following way:
\begin{equation*}\label{rel2}
(c',u_1, \ldots, u_n) \sim' (c'',v_1, \ldots, v_n) \iff  c'\mathcal{Z}(A)=c''\mathcal{Z}(A) \textrm{ and } u_i=v_i \textrm{ for every } i=1, \ldots n.
\end{equation*}
It is not hard to prove that $\sim'$ is an equivalence relation and that
\[
\lbrace (c,u_1, \ldots, u_n) \ | \ \textrm{satisfying } \eqref{sigmac}-\eqref{sigmauiuj} \rbrace/ \sim'\\
\quad \cong \quad
\lbrace (\sigma\varphi_c, d_1, \ldots, d_n) \ | \ \textrm{satisfying } \eqref{varphi^2}-\eqref{didj} \rbrace.
\]
\end{remark}

\section{Examples in low dimension}\label{examples}

\begin{example}
Let $n=0$. Then $A=Cl(\alpha)$ is the Clifford algebra with one generator $G$ such that $G^2=\alpha \neq 0$. We can identify $E(0)$ with the group algebra $k C_2=k\langle 1, g \rangle$ built on the cyclic group $C_2$ of order $2$. Notice that $A$ is a commutative algebra.

\autoref{evencase} tells us that in this case $A$ admits the following  types of $k C_2$-coactions:
\begin{itemize}
\item $\rho(a)= a \otimes 1$ for every $a \in A$,
\item $\rho(a)= a \otimes \frac{1 +g}{2}+\sigma(a) \otimes \frac{1-g}{2} $ for every $a \in A$.
\end{itemize}
Notice that both these coactions do not depend explicitly on an invertible element $c \in A$, since $A$ is commutative and $c^{-1}ac=a$ for every $a \in A$. If we let $k=\mathbb{R}$ we can recover some acknowledged facts. It is well-known that if $\alpha<0$, then $A \cong \mathbb{C}$ and if $\alpha>0$, then $A \cong \mathbb{D}= \mathbb{R} \oplus j\mathbb{R}$, where $j^2=1$ (see, e.g., \cite[p. 217]{Lo}). $\mathbb{D}$ is usually called \emph{the algebra of split-complex numbers} (or \emph{hyperbolic numbers}). The classification of $kC_2$-coactions exhibited above agrees with the well-known fact that
\[\textrm{Aut}_{\mathbb{R}}(\mathbb{C})=\lbrace \textrm{Id}_{\mathbb{C}}, \sigma\rbrace, \ \textrm{Aut}_{\mathbb{R}}(\mathbb{D})=\lbrace \textrm{Id}_{\mathbb{D}}, \sigma \rbrace.\]
\end{example}

\begin{example}\label{4dimexample}
Let $n=1$. Then $A=Cl(\alpha, \beta, \gamma)$ is the Clifford algebra with generators $G$, $X$ such that $G^2=\alpha$, $X^2=\beta$ and $GX+XG =\gamma$. We assume that $\gamma^2-4\alpha\beta \neq 0$, i.e. that $A$ is simple. Remember that we can always find two generators $G'$ and $X'$ of $A$ such that $G'^2=\alpha'$ and $X'^2=\beta'$ and $\alpha'\beta'=\gamma^2-4\alpha \beta\neq 0$: $A$ is a \emph{generalized quaternion algebra} over $k$.

In this case, the Hopf algebra $H:=E(1)$ is Sweedler's Hopf algebra and \autoref{oddcase} tells us that $A$ admits the following  types of $H$-coactions:
\begin{equation}\label{quatcoaction}
\rho(a)=   a \otimes \frac{1 +g}{2}+c^{-1}ac \otimes \frac{1 -g}{2} -  (ua-c^{-1}acu) \otimes \frac{x+gx}{2}+(uc^{-1}ac-au) \otimes \frac{x -gx}{2},
\end{equation}
where $c$ is an invertible element of $A$ such that $c^2 \in k$ and $u$ is an element of $A$ such that $u^2 \in k$ and $uc+cu=0$.

Let us define the set
\[k^{\frac{1}{2}}:=\lbrace a \in A \ | \ a^2 \in k \rbrace.\]
Clearly $k \subseteq k^{\frac{1}{2}}$ and when $c \in k$, then $uc+cu=0 \iff u =0$. Then we find the trivial coaction defined by $\rho(a)=   a \otimes 1$ for every $a \in A$. The set $k^{\frac{1}{2}} \setminus k$ is always non-empty (as $G^2=\alpha \in k$, $X^2=\beta \in k$) and is sometimes called \emph{the set of pure quaternions} of $A$ (cf. \cite[Prop. III.1.3]{La}).

Fix an invertible pure quaternion $c \in k^{\frac{1}{2}} \setminus k$. It is not hard to check that an element $u \in k^{\frac{1}{2}}$ such that $uc+cu=0$ is precisely another pure quaternion orthogonal to $c$. Therefore any (non-trivial) $H$-coaction on $A$ is given by \eqref{quatcoaction} where $c$ is an invertible pure quaternion and $u \in c^{\perp}$.

Furthermore, one can easily prove that when $c=-\gamma+2GX$ and $u=Gc^{-1}=\frac{1}{\gamma^2-4\alpha\beta}(-\gamma G+2\alpha X)$ we find the canonical coaction
\[\rho(a)=   a \otimes \frac{1 +g}{2}+\sigma(a) \otimes \frac{1 -g}{2} +  c^{-1}(Ga-aG) \otimes \frac{x+gx}{2}+(Ga-aG)c^{-1} \otimes \frac{x -gx}{2},\]
i.e.
\[\rho(G)=   G \otimes g, \quad \rho(X)=  X \otimes g +  1 \otimes x.\]
\end{example}

\begin{example}
Let $n=2$, $k=\mathbb{C}$, $\alpha=\beta_1=\beta_2=1$, $\gamma_1=\gamma_2=\lambda_{12}=0$. Then $A=Cl_3(\mathbb{C})$ is the eight-dimensional Clifford algebra over the complex numbers with pseudoscalar $z=GX_1X_2$. \autoref{evencase} tells us that each $E(2)$-coaction on $A$ is in correspondence with a triple $(c,u_1,u_2)$ such that $c$ is invertible, $u_1^2,u_2^2,u_1u_2+u_2u_1 \in \mathbb{C}\oplus\mathbb{C}z$ and either 
\begin{enumerate}
\item $c^2 \in \mathbb{C}\oplus\mathbb{C}z$ and $cu_1+u_1c=0=cu_2+u_2c$,
\item $c\sigma(c) \in \mathbb{C}\oplus\mathbb{C}z$ and $c\sigma(u_1)+u_1c=0=c\sigma(u_2)+u_2c$.
\end{enumerate}

Remember that the algebra $A$ is isomorphic to the direct product $A_0 \times A_0$ \--- where $A_0 \cong Cl_2(\mathbb{C})$ is the even subalgebra of $A$ \---, via the isomorphism  $f: A \rightarrow A_0 \times A_0$ defined by
\begin{eqnarray*}
f(G)&=&\left(-iX_1X_2,iX_1X_2 \right),\\
f(X_1)&=&\left(iGX_2, -iGX_2 \right),\\
f(X_2)&=&\left(-iGX_1,iGX_1 \right).
\end{eqnarray*}
By means of $f$ we obtain that each $E(2)$-coaction on $A$ is determined by a tuple $(c_1,c_2,r_1,r_2,s_1,s_2)$ of elements of $A_0$ such that $c_1$ and $c_2$ are units of $A_0$, the elements $r_1^2,r_2^2,s_1^2,s_2^2,r_1s_1+s_1r_1, r_2s_2+s_2r_2$ all belong to $\mathbb{C}$ and either 
\begin{enumerate}
\item $c_1^2,c_2^2 \in \mathbb{C}$ and $c_1r_1+r_1c_1=0=c_2r_2+r_2c_2$, $c_1s_1+s_1c_1=0=c_2s_2+s_2c_2$,\label{firstcase}
\item $c_2=\lambda c_1^{-1}$ for some $\lambda \in \mathbb{C}^{\times}$ and $r_2=-c_1^{-1}r_1c_1$, $s_2=-c_1^{-1}s_1c_1$.\label{secondcase}
\end{enumerate}
Let us split the two cases.

Since $A_0 \cong Cl_2(\mathbb{C})\cong Cl(1,1,0)$, a tuple $(c_1,c_2,r_1,r_2,s_1,s_2)$ that satisfies \eqref{firstcase} determine four pairs $(c_1,r_1)$, $(c_1,s_1)$, $(c_2,r_2)$, $(c_2,s_2)$ each characterizing an $E(1)$-coaction on $A_0$ (see Example~\ref{4dimexample}). Then we find that tuples satisfying \eqref{firstcase} are given by

\medskip

\begin{center}
\tabulinesep=0.2em
\begin{tabu}{| c | c |}
\hline
Tuple & Conditions  \\
\hline
$(c_1,c_2,0,0,0,0)$ & $c_1, c_2 \in \mathbb{C}^{\times}$.\\
\hline
$(c_1,c_2,0,r_2,0,s_2)$ & \makecell{$c_1 \in \mathbb{C}^{\times}$, $c_2$ is an invertible pure quaternion in $A_0$,\\
$r_2,s_2$ are pure quaternions contained in $c_2^{\perp}$.}\\
\hline
$(c_1,c_2,r_1,0,s_1,0)$ & \makecell{$c_1$ is an invertible pure quaternion in $A_0$, $c_2 \in \mathbb{C}^{\times}$,\\
$r_1,s_1$ are pure quaternions contained in $c_1^{\perp}$.} \\
\hline
$(c_1,c_2,r_1,r_2,s_1,s_2)$ & \makecell{$c_1, c_2$ are invertible pure quaternions in $A_0$,\\  $r_1,s_1$ are pure quaternions contained in $c_1^{\perp}$,\\
$r_2,s_2$ are pure quaternions contained in $c_2^{\perp}$.} \\
\hline
\end{tabu}
\end{center}

\medskip

On the other hand, a tuple that satisfies \eqref{secondcase} must be of the form 
\[(c_1,\lambda c_1^{-1},r_1,-c_1^{-1}r_1c_1,s_1,-c_1^{-1}s_1c_1),\] with $c_1 \in \textrm{U}(A_0)$, $\lambda \in \mathbb{C}^{\times}$, $r_1^2,s_1^2, r_1s_1+s_1r_1 \in \mathbb{C}$. Hence we can easily conclude that tuples satisfying \eqref{secondcase} are the following

\medskip

\begin{center}
\tabulinesep=0.2em
\begin{tabu}{| c | c |}
\hline
Tuple & Conditions  \\
\hline
$(c_1,\lambda c_1^{-1},r_1,-r_1,s_1,-s_1)$ & $c_1 \in \textrm{U}(A_0)$, $\lambda,r_1,s_1 \in \mathbb{C}^{\times}$.\\
\hline
$(c_1,\lambda c_1^{-1},r_1,-r_1,0,0)$ & $c_1 \in \textrm{U}(A_0)$, $\lambda, r_1 \in \mathbb{C}^{\times}$.\\
\hline
$(c_1,\lambda c_1^{-1},0,0,s_1,-s_1)$ & 
$c_1 \in \textrm{U}(A_0)$, $\lambda,s_1 \in \mathbb{C}^{\times}$. \\
\hline
$(c_1,\lambda c_1^{-1},r_1,-c_1^{-1}r_1c_1,s_1,-c_1^{-1}s_1c_1)$ & \makecell{$c_1 \in \textrm{U}(A_0)$, $\lambda \in \mathbb{C}^{\times}$,\\ $r_1,s_1$ are pure quaternions.} \\
\hline
\end{tabu}
\end{center}
\end{example}
\medskip

\noindent\textbf{Acknowledgements.}
FR would like to thank Claudia Menini and Blas Torrecillas for their support during his research work and for comments that helped shape the present article. The author was partially supported by MUR within the National Research Project PRIN 2017.

\end{document}